\documentclass[11pt]{amsart}

\renewcommand{\arraystretch}{1.2}
\usepackage[margin=2.8cm]{geometry}
\usepackage{latexsym}
\usepackage{amsxtra}
\usepackage{verbatim}
\usepackage{color}
\usepackage{amsthm,amsmath,amsfonts,amssymb,mathrsfs,amscd,graphics}
\usepackage{amssymb}
\usepackage{appendix}
\usepackage{ifpdf}
\usepackage{enumerate}
\usepackage{fancyhdr}
\usepackage{mathrsfs}
\usepackage[mathscr]{eucal}
\usepackage{upgreek}
 \usepackage{wasysym}
\usepackage{esint}
\usepackage{pstricks}
\usepackage{color}
\usepackage{cancel}
\usepackage{verbatim}
\usepackage[all]{xy}

\usepackage[colorlinks=true]{hyperref}

\usepackage{color}

\usepackage{lipsum}
\usepackage{float}
\usepackage{mathtools}

\usepackage{longtable}
\allowdisplaybreaks
\usepackage{tikz}
\usepackage{tikz-cd}
\usepackage{amstext, amsbsy}
\usepackage{amsfonts, extarrows}
\newtheorem{theorem}{Theorem}[section]
\newtheorem{lemma}[theorem]{Lemma}
\theoremstyle{definition}
\newtheorem{definition}[theorem]{Definition}

\newtheorem{example}[theorem]{Example}

\newtheorem{proposition}[theorem]{Proposition}
\newtheorem{corollary}[theorem]{Corollary}
\newtheorem{conjecture}[theorem]{Conjecture}
\newtheorem{assumption}[theorem]{Assumption}
\theoremstyle{remark}
\newtheorem{remark}[theorem]{Remark}
\numberwithin{equation}{section}

  \newcommand{\cF}{{\mathcal F}}
  
    \newcommand{\sE}{{\mathscr E}}
   \newcommand{\cM}{{\mathcal M}}

    \newcommand{\cW}{{\mathcal W}}

    \newcommand{\cH}{{\mathcal H }}

    \newcommand{\cA}{{\mathcal A}}
        \newcommand{\cD}{{\mathcal D}}
            
  \newcommand{\C}{{\mathbb C}}
  
  \newcommand{\Z}{{\mathbb Z}}
  \newcommand{\Q}{{\mathbb Q}}

\renewcommand{\H}{{\mathcal  H}}

\newcommand{\Jac}{{\rm  {Jac}}}
\newcommand{\cL}{\mathcal L}
\newcommand{\vir}{\mathrm{vir}}

\newcommand{\bs}{\mathbf{s}}
\newcommand{\bt}{\mathbf{t}}
\newcommand{\bq}{\mathbf{q}}
\newcommand{\bx}{\mathbf{x}}
\newcommand{\wt}{\widetilde{t}}

\newcommand{\pop}{\partial\over\partial}

\newcommand{\one}{\mathbf{1}}
\newcommand{\ogamma}{\vec{\gamma}}
\newcommand{\wpsi}{\bar\psi}
\newcommand{\LL}{\llangle[\big]}
\newcommand{\RR}{\rrangle[\big]}

\newcommand{\LD}{\Big\langle}
\newcommand{\RD}{\Big\rangle}

\makeatletter
\newsavebox{\@brx}
\newcommand{\llangle}[1][]{\savebox{\@brx}{\(\m@th{#1\langle}\)}%
  \mathopen{\copy\@brx\kern-0.5\wd\@brx\usebox{\@brx}}}
\newcommand{\rrangle}[1][]{\savebox{\@brx}{\(\m@th{#1\rangle}\)}%
  \mathclose{\copy\@brx\kern-0.5\wd\@brx\usebox{\@brx}}}
\makeatother

  \newcommand{\<}{\langle}
  \renewcommand{\>}{\rangle}

\usepackage{tikz}
\usepackage[framemethod=tikz]{mdframed}
\usetikzlibrary{shapes,decorations}
\usetikzlibrary{arrows,matrix,positioning}

\usepackage[most]{tcolorbox}
\usepackage{color}

\usepackage{amsfonts,mathrsfs}

\begin{document}
\title{Virasoro constraints in Quantum Singularity Theories}
\date{}


%    author one information

 \author{Weiqiang He and Yefeng Shen}
 %, with an appendix by A. Ellis Francis}

\maketitle

\begin{abstract}
We introduce Virasoro operators for any Landau-Ginzburg pair $(W, G)$ where $W$ is a non-degenerate quasi-homogeneous polynomial and $G$ is a certain group of diagonal symmetries.
%Using Givental's formula, these operators satisfy Virasoro relations.
We propose a conjecture that the total ancestor potential of the FJRW theory of the pair $(W,G)$ is annihilated by these Virasoro operators.
We prove the conjecture in various cases, including:
(1) invertible polynomials with the maximal group,
(2) some two-variable %homogeneous Fermat
polynomials with the minimal group,
(3) certain Calabi-Yau polynomials with groups.
We also discuss the connections among Virasoro constraints, mirror symmetry of Landau-Ginzburg models, and Landau-Ginzburg/Calabi-Yau correspondence. % and Hodge integrals.
\end{abstract}

{
\hypersetup{linkcolor=black}
\setcounter{tocdepth}{2} \tableofcontents
}

\setcounter{section}{-1}
\section{Introduction}
%\textcolor{red}{Something about the history of Virasoro constraints.}
Virasoro constraints have been proposed in Gromov-Witten theory by Eguchi, Hori, and Xiong \cite{EHX} and S. Katz \cite{Katz}.
It becomes one of the most fundamental and fascinating conjectures in Gromov-Witten theory.
Despite significant developments in the literature \cite{EHX, EJX, LT, GP, DZ, G-ham, OP}, it remains as one of the most difficult conjectures in Gromov-Witten theory.
Virasoro constraints have also been proposed in various topics in enumerative geometry \cite{DZ, JT, CGT, OOP, MOOP}.

In this paper, we propose Virasoro constraints for certain Landau-Ginzburg pairs $(W, G)$, with $W$ a certain quasi-homogeneous polynomial and $G$ a certain group of symmetries of $W$.
For such a LG pair $(W, G)$, Fan, Jarvis, and Ruan constructed a  {\em Cohomological Field Theory} (in the sense of Kontsevich and Manin \cite{KM}) in \cite{FJR1, FJR}, based on a proposal of Witten \cite{Wit-eq}.
Originally it was called {\em quantum singularity theory} in \cite{FJR}.
Nowadays the theory constructed in \cite{FJR1, FJR} are widely called {\em Fan-Jarvis-Ruan-Witten theory} (or FJRW for short). The construction used both algebraic and analytic tools.
%  %\footnote{The terminology is introduced by Kontsevich and Manin in \cite{KM}.},
There are purely algebraic constructions of Cohomology field theories for the same LG pair by Polishchuk-Vaintrob \cite{PV}, and Kiem-Li \cite{KL}.
The three constructions are conjectured to be equivalent.
These theories are {\em LG A-model theories} in the study of mirror symmetry \cite{IV, BH, BHe, Kr, CIR, LLSS, HLSW, HPSV}.

\subsection{Admissible Landau-Ginzburg pairs and the state spaces}
%We work over the field of complex numbers $\mathbb{C}$.
Following the work of Fan, Jarvis, and Ruan \cite{FJR}, we let
$W: \mathbb{C}^n\to \mathbb{C}$  be a {\em non-degenerate quasi-homogeneous polynomial} of $n$ variables. Here the non-degeneracy means the $W$ has isolated critical points only at the origin and the weights (or degrees) of the variables are uniquely given by rational numbers.
We write
$${\rm wt}(x_i):=q_i\in\big(0, {1\over 2}\big]\cap \mathbb{Q}.$$
%There exist $w_1\cdots, w_n, d\in \mathbb{N}$, such that $q_i:=w_i/d$.

%Take $\mathbb{G}_m={\rm GL}_1(\C)$ and
Let $G_W\leq \mathbb{G}_m^n$ be the {\em group of diagonal symmetries} of $W$ defined by
\begin{equation*}
\label{maximal-group}
G_W:=\Big\{ (\lambda_1, \cdots, \lambda_n)\in \mathbb{G}_m^n\bigg| W(\lambda_1\, x_1, \cdots, \lambda_n\, x_n)=W(x_1, \cdots, x_n)\Big\}.
\end{equation*}
There exists the {\em exponential grading element}
\begin{equation*}
\label{expo-element}
J_W:=(\exp(2\pi  \sqrt{-1} q_1), \cdots, \exp(2\pi \sqrt{-1} q_n))\in G_W.
\end{equation*}
\begin{definition}
[Admissible LG pairs]
\label{admissible-pair}
A subgroup $G\leq G_{W}$ is called {\em admissible} if $J_W\in G$.
\begin{itemize}
\item
We call the pair $(W,G)$ an {\em admissible LG pair} if $G$ is admissible.
\item We call $\<J_W\>$ the {\em minimal group} of $W$.
\item We call $G_W$ the {\em maximal group} of $W$ and write $G_W=G_{\rm max}$ or $G_W={\rm Aut}(W)$ sometimes.
\end{itemize}
\end{definition}
For an admissible LG pair $(W, G)$, there exists a super vector space with a non-degenerate bilinear pairing and a bigrading \cite{IV, BH, BHe, Kr}.
We call it the {\em state space} and denote it by $\cH_{W, G}$.
Let us briefly review the construction here.
\subsubsection{A state space $\cH_{W, G}$}
For each $\gamma\in G\leq \mathbb{G}_m^n$,
we write
\begin{equation}
\label{group-exponent}
\gamma:=\left(\exp(2\pi\sqrt{-1}\theta_1), \cdots, \exp(2\pi\sqrt{-1}\theta_n)\right),
\end{equation}
where for all $i$ such that $1\leq i\leq n$, $\theta_i\in [0, 1)\cap\Q$ are given uniquely.
Define the {\em degree shift number} of $\gamma$ to be
\begin{equation}
\label{degree-shift}
\iota_\gamma:={\rm age}(\gamma)-\sum_{i=1}^{n}q_i:=\sum_{i=1}^n(\theta_i-q_i).
\end{equation}
The fixed locus of $\gamma$ in $\C^n$ is denoted by ${\rm Fix}(\gamma)$.
It is a subspace of $\C^n$ of dimension
$$N_\gamma:=\dim_{\C}{\rm Fix}(\gamma).$$
Let $W_\gamma:=W|_{{\rm Fix}(\gamma)}$ be the restriction.
We denote $d{\bf x}_\gamma$ the standard top form on ${\rm Fix}(\gamma)$. %of $W$ on the fixed locus of $\gamma$.
Recall the {\em Jacobian algebra} of $W$ is given by $${\rm Jac}(W):=\C[x_1, \cdots, x_n]\bigg/\left({\partial W\over \partial x_1}=\cdots={\partial W\over \partial x_n}=0\right).$$
\begin{definition}
[A state space]
For an LG pair $(W,G)$, the {\em state space} $\cH_{W, G}$ is defined to be the direct sum of $G$-invariant spaces
$$\cH_{W, G}:=\bigoplus_{\gamma\in G}\cH_\gamma, \quad \cH_\gamma:=\left({\rm Jac}(W_\gamma)\cdot d{\bf x}_\gamma \right)^G.$$
Here $\gamma=(\lambda_1, \cdots, \lambda_n)\in G$ acts on $x_i$ and $dx_i$ both by multiplying $\lambda_i$.
We call an element in the subspace $\cH_\gamma$ \emph{narrow} if $N_\gamma=0$, otherwise it is \emph{broad}.
\end{definition}

Following the notation of \cite{Kr}, %in the remaining part of this paper, %we use $\alpha|\gamma\>$  to denote an element of $\H(W, G)$, where $\gamma\in G$ and $\alpha\in \Jac(W_\gamma)$.
the state space $\cH_{W, G}$ is spanned by elements of the form $\alpha|\gamma\rangle$, where $\gamma\in G$ and $\alpha\in {\rm Jac}(W_\gamma)\cdot d{\bf x}_\gamma$.
%According to \cite{IV, Kr}, the vector space $\cH_{W, G}$ is bigraded.
If $x_i$ is fixed by $\gamma$, i.e. $\theta_i=0$, then we write $i\in W_\gamma$.
Consider a monomial representative
$$\alpha=[{\bf x}_\gamma^{\bf m}]\cdot d{\bf x}_\gamma=\left[\prod_{i\in W_\gamma} x_i^{m_i}\right]\cdot d{\bf x}_\gamma\in {\rm Jac}(W_\gamma)\cdot d{\bf x}_\gamma.$$
%be represented by a monomial  $${\bf x}_\gamma^{\bf m}=\prod_{\theta_i=0} x_i^{m_i}.$$
Let ${\rm wt}(\alpha)$ be the weight of the form ${\bf x}_\gamma^{\bf m}d{\bf x}_\gamma$, defined by
$${\rm wt}(\alpha):=\sum_{i\in W_\gamma}(m_i+1)q_i.$$

\subsubsection{A pairing on $\cH_{W, G}$}
%The Grothendieck residue pairing  $${\rm Res}_{W}: {\rm Jac}(W)\times {\rm Jac}(W)\to \C$$ on ${\rm Jac}(W)$ induces a pairing on the vector space $\cH_{W, G}$.
We fix  a homogeneous basis of $\cH_{W,G}$, denoted by $$\Big\{\phi_a=[f_a]d\bx_{\gamma_a}|\gamma_a\>\mid \gamma_a\in G, [f_a]\in {\rm Jac}(W_{\gamma_a})\Big\}.$$ % and $\phi_b=[f_b]d\bx_{\gamma_b}$
Let ${\rm Res}_{W_{\gamma_a}}$ be the Grothendieck residue pairing on ${\rm Jac}(W_{\gamma_a})$.
Then there is a natural nondegenerate bilinear pairing $(\cdot , \cdot )$ on $\cH_{W,G}$, defined by
\begin{equation}
\label{qst-pairing}
\eta_{ab}:=(\phi_a, \phi_b):=
\left\{
\begin{array}{ll}
0, & \textit{if}\quad \gamma_a\neq\gamma_b^{-1};\\
{\rm Res}_{W_{\gamma_a}}(f_a f_b \cdot d{\bx}_{\gamma_a}), & \textit{if}\quad \gamma_a=\gamma_b^{-1}.\\
\end{array}
\right.
%:=([f_a]\bx_{\gamma_a}, [f_b]\bx_{\gamma_b})
%:=\delta_{a}^{b}\,{\rm Res}_{W_\gamma}(f_a, f_b).}
\end{equation}

\subsubsection{Bigrading and parity}
We denote the {\em central charge} of the polynomial $W$ by
\begin{equation}
\label{central-charge}
\widehat{c}_W:=\sum_{i=1}^{n}(1-2q_i).
\end{equation}
Following \cite{IV, Kr}, we introduce a bigrading and a parity on $\cH_{W, G}$.
%\begin{definition}[Bigrading and parity]\label{bigrad}
%Let $\widehat{c}_W$ be the central charge of $W$.
For an homogeneous element $$\phi_a:=\alpha|\gamma\>=[{\bf x}_\gamma^{\bf m}]\cdot d{\bf x}_\gamma|\gamma\>\in \cH_{W, G},$$ we assign a bigrading $(\mu_a^+,\mu_a^-)$,
\begin{equation}
\label{bigrading}
\left\{
\begin{array}{ll}
\mu_a^+:={\rm wt}(\alpha)+\iota_\gamma-{\widehat{c}_W\over 2}; \\%\deg\alpha+\sum_{\theta_i\neq 0}(\theta_i-q_i)-{\widehat{c}_W\over 2};\\
\mu_a^-:=N_{\gamma}-{\rm wt}(\alpha)+\iota_\gamma-{\widehat{c}_W\over 2}.%\deg\alpha+\sum_{\theta_i\neq 0}(\theta_i-q_i)+\sum_{\theta_i=0}(1-2q_i)-{\widehat{c}_W\over 2}.
\end{array}
\right.
\end{equation}
%Here  $\widehat{c}_W$ is the central charge of $W$.
%$$\mu^+:=\deg\alpha+\iota_\gamma-\widehat{c}_W/2, \quad \mu^-:=N_\gamma-\deg\alpha+\iota_\gamma-\widehat{c}_W/2.$$
%%%%%%%%%%The bigraded vector space $H_{W, G}$ has a parity.
and a parity
\begin{equation}
\label{parity}
|\phi_a|=(-1)^{N_\gamma}.
\end{equation}
We also define a complex degree
\begin{equation}
\label{complex-degree}
\deg_\C\phi_a:={1\over 2}\left(\mu_a^++\mu_a^-+\widehat{c}_W\right).
\end{equation}
%an element $\phi:=\alpha|\gamma\rangle\in \cH_{W, G}$ is called {\em even}  if the complex dimension $N_\gamma$ is even; it is called {\em odd} if  $N_\gamma$ is odd.
%We denote the parity of $\phi$ by $|\phi|=(-1)^{N_\gamma}.$
%\end{definition}
%In mirror symmetry of LG models, the bigrading and parity here are assigned for the A models space $\cH_{W,G}$.
%We will also discuss B-model bigrading and parity in Section \ref{sec-max}.
%We will always consider $(W, G)$ as an admissible LG A-model pair unless otherwise stated.

\iffalse
There is another $\Z_2$ grading on the state space $\H(W, G)$, which is defined as following:
\begin{definition}\label{Z_2}
	For $\phi_a=\alpha|\gamma\>\in \H(W, G)$, the $Z_2$ grading $|\cdot|$ is defined as
	$$|\phi_a|:={\rm dim_{\C} Fix}(\gamma), \quad {\rm mod}\, 2.$$
\end{definition}
\fi

\begin{example}\label{bigrforell}
For the Fermat cubic pair $(W=x_1^3+x_2^3+x_3^3, G=\<J_W\>),$
we have the following list for a basis of $\cH_{W,G}$.
This example will be discussed in Section \ref{fermat-cubic} with more details.

\begin{table}[H]
\caption{The state space of a Fermat cubic}
\begin{center}
\begin{tabular}{|c||c|c|c|c|}
  \hline
 % element
  $\phi_a$ &  $1|J\>$ &    $1|J^2\>$  & $d{\bf x}|J^0\rangle$ &   $x_1x_2x_3d{\bf x}|J^0\rangle$\\
  \hline
 % bigrading
 $(\mu_a^+,\mu_a^-)$ & $(-{1\over 2}, -{1\over 2})$ & $({1\over 2}, {1\over 2})$   & $(-{1\over 2}, {1\over 2})$ & $({1\over 2}, -{1\over 2})$ \\
  \hline
 % parity
 $|\phi_a|$ & $1$ & $1$ & $-1$ & $-1$\\
   \hline
  %degree
 $\deg_\C\phi_a$ & $0$ & $1$ & ${1\over 2}$ & ${1\over 2}$\\
  \hline
\end{tabular}
\end{center}
\label{table-cubic}
\end{table}
\end{example}

%As a consequence, the pairing is non-degenerate \cite{}.
%For a fixed basis $\{\phi_a\}$, we denote the entries of the pairing matrix by $$\eta_{ab}:=(\phi_a, \phi_b).$$
%and denote the inverse of the pairing matrix by $$\eta^{ab}:=(\eta^{-1})_{ab}.$$

%For $(W, G)$, $W$ is an invertible polynomial\footnote{Not necessary to be invertible}, $G$ is its admissible symmetry group. The state space of LG A (or B)  model is
%$$\H(W, G):=\bigoplus_{\gamma\in G}\left(\Jac(W_{\gamma})dx_{\gamma}\right)^G.$$
%Following the notation of \cite{Kr}, in the remaining part of this paper, we use $\alpha|\gamma\>$  to denote an element of $\H(W, G)$, where $\gamma\in G$ and $\alpha\in \Jac(W_\gamma)$.

%For simplicity, we write $\gamma=(\Theta_1, \ldots, \Theta_N)$, $0\leq \Theta_i<1$.

\subsection{Virasoro operators for admissible LG pairs}
%check wiki and OP's notation.
%{\red Prefer Pochhammer symbols} % as $\Gamma(-N)$ does not make sense.}
For $\ell\in \mathbb{R}$ and $n\in\mathbb{Z}_{\geq 0}$, we denote the Pochhammer symbol by
\begin{equation}
\label{pochhammer}
(\ell)_n:=
\begin{dcases}
\ell(\ell+1)\cdots (\ell+n-1), &\text{ if } n\geq 1.\\%{(m+k-1)!\over (\ell-1)!}.
1, & \text{ if }  n=0.
\end{dcases}
\end{equation}
%By convention $(\ell)_0:=1$.
We assign a variable $t_{m}^{a}$ to each $\phi_a z^m\in \cH_{W,G}[\![z]\!]$.
\begin{definition}
[Virasoro operators]
\label{def-virasoro}
For each integer $k\in\mathbb{Z}_{\geq -1}$, we introduce a differential operator
\begin{align}
L_k:=&-\left(\frac{3-\widehat{c}_W}{2}\right)_{k+1} {\pop t_{k+1}^{0}}\nonumber\\
%-\frac{\Gamma(\frac{5-c_W}{2}+k)}{\Gamma(\frac{3-c_W}{2})}\partial_{k+1, 0}\nonumber\\
&+\sum_{m=0}^\infty\left(\mu^+_a+m+\frac{1}{2}\right)_{k+1} t^a_m{\pop t_{m+k}^{a}}\nonumber\\
%&+\sum_{m=0}^\infty\frac{\Gamma(\mu^+_a+m+k+\frac{3}{2})}{\Gamma(\mu^+_a+m+\frac{1}{2})}t^a_m\partial_{m+k, a}\nonumber\\
&+\frac{\hbar^2}{2}\sum_{m=-k}^{-1}(-1)^m\left(\mu^+_a+m+\frac{1}{2}\right)_{k+1}  \eta^{ab}{\pop t_{-m-1}^{a}}{\pop t_{m+k}^{b}}\label{virasoro-operator}\\
&+\frac{1}{2\hbar^2}\delta_{-1,k}\eta_{ab}t_0^at_0^b\nonumber\\
&-{\delta_{0, k}\over 4}\sum_{a}(-1)^{|\phi_a|}(\mu_a^+-{1\over 2})(\mu_a^++{1\over 2})\nonumber.
%+\delta_{0, k}\cdot {\chi\over 24}\cdot \left({3-\widehat{c}_W\over 2}\right).\nonumber
\end{align}
%If $k=-1$, we have convention $\prod\limits_{j=0}^{k}(\cdots)=1$.
%\end{definition}
%We define Virasoro operator $\{L_k, k\geq -1\}$ similar as \cite{G}.
\iffalse
\begin{definition}
\label{def-virasoro}
Let $\Gamma(x)$ be the Gamma function. For $k\geq -1$, define $L_k$ as: \textcolor{red}{The leading term is from the dilaton shift $q^a_m=t^a_m-\delta^{a, 0}\delta_{k, 1}$, and $\mu_0^+=-\frac{c_W}{2}$. I think that $\mu_a$ should be $k$ in the formula of p. 5\cite{G}.}
\begin{eqnarray}
L_k&:=&-\frac{\Gamma(\frac{5-c_W}{2}+k)}{\Gamma(\frac{3-c_W}{2})}\partial_{k+1, 0}+\sum_{m=0}^\infty\frac{\Gamma(\mu^+_a+m+k+\frac{3}{2})}{\Gamma(\mu^+_a+m+\frac{1}{2})}t^a_m\partial_{m+k, a}\nonumber\\
&&+\frac{\hbar}{2}\sum_{m=-k}^{-1}(-1)^m\frac{\Gamma(\mu^+_a+m+k+\frac{3}{2})}{\Gamma(\mu^+_a+m+\frac{1}{2})}\eta^{ab}\partial_{-m-1, a}\partial_{m+k, b}\label{virasoro-operator}\\
&&+\frac{1}{2\hbar}\delta_{-1,k}\eta_{ab}t_0^at_0^b
-{\delta_{0, k}\over 4}\sum_{a}(-1)^{|a|}(\mu_a^+-{1\over 2})(\mu_a^++{1\over 2})\nonumber.
%+\delta_{0, k}\cdot {\chi\over 24}\cdot \left({3-\widehat{c}_W\over 2}\right).\nonumber
\end{eqnarray}
\end{definition}
\fi
Here $\eta^{ab}$ is the $(a, b)$-th entry of the inverse matrix of the pairing matrix $(\eta_{ab})$.
\end{definition}
Following the work of Givental \cite{G-ham},
these differential operators are related to quantization of quadratic Hamiltonians.
This allows us to verify
\begin{proposition}\label{prop-virasoro}
The differential operators $\{L_k\}_{k\geq -1}$ satisfy the relations:
\begin{equation}
\label{Virasoro-relation}
[L_m, L_n]=(m-n)L_{m+n}.
\end{equation}
\end{proposition}
In general, a set of operators $\{L_k\}_{k\in\mathbb{Z}}$ are called {\em Virasoro operators} if
$$[L_m, L_n]=(m-n)L_{m+n}+{m^3-m\over 12}\cdot\delta_{m+n,0}\cdot c,$$
with some constant $c$ (which is called the {\em central charge of the Virasoro algebra}).
The relations become \eqref{Virasoro-relation} if we restrict to $k\geq -1$.
By abuse of the notations, we will call $\{L_k\}_{k\geq -1}$ in \eqref{virasoro-operator} the {\em Virasoro operators for admissible LG (A-model) pairs}.

\subsection{Virasoro constraints}

The essential ingredients in each construction \cite{FJR1, FJR, PV, KL} are moduli spaces of $W$-spin structures in \cite{FJR1, FJR}, or $G$-spin structures in \cite{PV} and \cite{KL}, over genus-$g$ orbifold curves decorated by the element
$$\ogamma=(\gamma_1, \cdots, \gamma_k)\in G^k.$$
We denote the moduli space by $\cW_{g, \ogamma}^G$. The moduli space carries a virtual fundamental classes, which will be denoted by $[\cW_{g, \ogamma}^G]^{\vir, \clubsuit}$.
%which is a generalization of Witten's top Chern class for $r$-spin curves.

Despite the different nature of the techniques, each construction gives a CohFT $\{\Lambda^{\clubsuit, W, G}_{g,k}\}$ on the isomorphic underlying state space $\cH_{W, G}$. Here $g$ and $k$ satisfies the stability condition  $2g-2+k>0$.
We denote the CohFTs by $\Lambda^{\clubsuit}$, where
$$\clubsuit={\rm FJRW}, {\rm PV}, {\rm KL}$$
stands for the abbreviation of each theory.

We will mainly consider FJRW theory in our paper, but the setting also applies to the other two.
% which we sometimes denoted by $\cH$.
Choosing a set of elements $\{\phi_{i}\in \cH_{\gamma_i}\}$,
the linear maps
$$\Lambda^{\clubsuit, W, G}_{g,k}: \cH_{W, G}^{\otimes k}\to H^*(\overline{\cM}_{g,k}, \C)$$
produces intersection numbers, called {\em ancestor invariants}
\begin{equation}
\label{ancestor-inv}
\LD \prod_{i=1}^{k} \tau_{\ell_i}(\phi_{i})%, \cdots,  \tau_{\ell_k}(\phi_{a_k})
\RD_{g, k}^{\clubsuit, (W, G)}
=\int_{\overline{\cM}_{g,k}} \Lambda^{\clubsuit, W, G}_{g,k}(\phi_1\otimes \ldots \otimes \phi_k)\cdot\prod_{i=1}^{k}\psi_{i}^{\ell_i},
\end{equation}
where $\phi_i\in \cH_{W, G}$, $\psi_i$'s are the first chern classes of cotangent line bundles on coarse moduli of spaces $\overline{\cM}_{g,n}$.
%and
%$\pi: \cW_{g, \ogamma}^G\to \overline{\cM}_{g,n}$
%is the forgetful morphism.
%We assign variable $t_m^a$ to each insertion $\psi^k\phi_a$, and

%In this paper, we consider the so-called {\em total ancestor potential} in each Quantum Singularity Theory.
\subsubsection{The total ancestor potential}
We write $$\bt(z):=\sum_{m\geq 0}\sum_{a}t_{m}^{a}\phi_a z^m.$$
%This definition coincides with that in \cite[Section 1.3]{Ge}.
As in \cite[Section 1.3]{Ge}, we assign parity for $t_{m}^{a}$, which coincides with that of $\phi_a$, and moving $t_{m}^{a}$ across $\phi_b$ will give a sign, i.e.
$$t_{m}^{a}\phi_b=(-1)^{|t_{m}^{a}|\cdot|\phi_b|}\phi_bt_{m}^{a}.$$
The {\em total ancestor potential} is given by
\begin{equation}
\label{ancestor-qst}
\mathcal{A}^{\clubsuit}_{W,G}%(\mathbf{t})
:=\exp\left(\sum_{g}\hbar^{2g-2}\sum_{k}{1\over k!}\LD\prod_{i=1}^{k}\bt(\psi_i)
%\bt(\psi_1), \cdots, \bt(\psi_k)
\RD_{g,k}^{\clubsuit, (W, G)}\right).
\end{equation}

The three CohFTs constructed in the work of Fan-Jarvis-Ruan \cite{FJR1, FJR}, Polishchuk-Vaintrob \cite{PV}, and Kiem-Li \cite{KL} are conjectured to be equivalent.
In particular, the pairing in \eqref{qst-pairing} matches the intersection pairing of Lefschetz thimbles in FJRW theory \cite{FJR}, the perfect pairing on intersection homology in KL theory \cite{KL}, and the Mukai pairing on Hochchild homology of $\Gamma$-equivariant matrix factorizations of $W$ \cite{PV}.
In general, the equivalence is known when restricting to a subspace of {\em narrow elements} in $\cH_{W, G}$ \cite{CLL}.
Numerically, if we consider the total ancestor potentials, the equivalence are verified for more cases \cite{Gue, HLSW, HPSV}.
%At the moment, no counter examples has been found.

%The descendent invariants will be defined in \eqref{quantum-inv}.

\subsubsection{Virasoro conjecture for admissible LG pairs}
Now we propose a Virasoro conjecture for the admissible LG A-model pairs. %in Definition \ref{admissible-pair}.
\begin{conjecture}
[Virasoro conjecture for admissible LG A-model pairs]
\label{VC}
%Let $W$ be a quasihomogeneous polynomial with isolated critical point only at the origin and $G$ be an admissible group.
For any admissible LG pair $(W, G)$,
the total ancestor potentials $\mathcal{A}^{\clubsuit}_{W,G}$ in FJRW/PV/KL theory satisfy the Virasoro constraints
	$$L_k\mathcal{A}^{\clubsuit}_{W,G}=0, \quad \textit{for all } k\geq -1.$$
\end{conjecture}
%see Section 1.1.

\begin{remark}
When $k=-1$, the constraint $L_{-1}\mathcal{A}^{\clubsuit}_{W,G}=0$ always holds true  in each theory \cite{FJR, PV, KL} as it is equivalent to the {\em string equations}.
%\subsubsection{Virasoro constraint of $L_0$}
%Our first result is for the Virasoro operator $L_0$.
When $k=0$, the last term in \eqref{virasoro-operator} can be rewritten in terms of the central charge and Euler characteristic of the LG pair by a supertrace formula \eqref{hodge-rr}.
If $\widehat{c}_W=3$, the supertrace term ${\rm Str}(\theta^2-{1\over 4})$ still vanishes and the constraints by $L_0$ is equivalent to a {\em grading equation} \eqref{grading-equation}, which is determined by an Euler vector that arises in the quantum singularity theory. By the {\em dilaton equation} \eqref{dilaton-operator}, such an equivalence still holds when $\widehat{c}_W\neq3.$
\end{remark}

%\begin{theorem}\cite{FJR, PV, KL} The Virasoro constraint $L_{-1}\mathcal{A}^{\clubsuit}_{W,G}=0$ holds for all admissible LG pairs $(W,G)$ in the case of  $\clubsuit={\rm FJRW}, {\rm PV}$, or ${\rm KL}$. \end{theorem}

\subsection{Main results}
In this paper, we study Virasoro Conjecture \ref{VC} in various situations.
We will mostly focus on the FJRW theory.
Part (2) of Theorem \ref{thm-max} is stated for PV theory. The result is slightly stronger there than in the FJRW theory, due to some technique advantages.

%We will mostly focus on the FJRW theory, except that Theorem \ref{thm-low} and Part (1) of Theorem \ref{thm-max} are stated for KL theory and PV theory.
%In these situations, the results in KL theory and PV theory are stronger than the results in FJRW theory, due to some technique advantages.

%Another term in the Euler vector is related to the Hodge (bigrading) operators of the LG pair by a conjectured supertrace formula.
%We generalized the argument in \cite{ET} and prove the supertrace formula for all admissible LG pairs.

\iffalse

\textcolor{red}{As a consequence, we prove}
\begin{theorem}
\label{thm-low}
For any admissible LG pair $(W,G)$, if the quantum singularity theory satisfies Assumption \ref{alge cycle}, then $L_{0}\mathcal{A}^{\clubsuit}_{W,G}=0$.
\end{theorem}
We can not prove any quantum singularity theory satisfies Assumption \ref{alge cycle} now. Assumption \ref{alge cycle} is a stronger version of dimension axiom, and it is closed related to whether the CohFT is induced by an algebraic cycle. In Remark \ref{KL bigrad}, we will discuss Assumption \ref{alge cycle} a little bit for KL theory.
%For a general admissible LG pair $(W, G)$, if we replace the KL theory in Theorem \ref{thm-low} by either FJRW theory or PV theory, then the constraints is not known, as the properties of virtual cycles needed in the grading equation (see Assumption \ref{alge cycle}) are not known in the FJRW/PV theory in general.
\fi

\subsubsection{Virasoro constraints for semi-simple Frobenius manifolds}
%and LG mirror symmetry}
The genus zero invariants in \eqref{ancestor-inv} gives a Frobenius manifold in the sense of Dubrovin \cite{Du}.
If the Frobenius manifold is {\em generically semi-simple}, i.e., the Frobenius algebra is semi-simple at a generic point of the Frobenius manifold, the total ancestor potential is uniquely constructed from the Frobenius manifold by the famous Givental-Teleman formula in \eqref{gt-formula} \cite{G-higher, T}.
By writing the Virasoro operators using quantization operators of certain quadratic Hamiltonians, % and investigating their commutator relations,
Givental proved the formula \eqref{gt-formula} after a modification % total ancestor potential with an underlying semi-simple Frobenius manifold
is annihilated by Virasoro operators \cite[Theorem 7.7]{G-ham}.
%This is Givental Theorem \ref{semi-simple-Virasoro}.
%\textcolor{red}{Also Dubrovin-Zhang?}

By applying Givental's result, we show the LG A-models of two types of admissible LG pairs have generically semisimple Frobenius manifolds and thus Virasoro Conjecture \ref{VC} hold for them. Both types are from invertible polynomials, which are
%One particular type of generically semi-simple Frobenius manifolds in quantum singularity theories are found using mirror symmetry between LG models for invertible polynomials.
%Invertible polynomials are
of the form $\sum_{i=1}^{n}\prod_{j=1}^{n}x_j^{a_{ij}},$
with a matrix $\left(a_{ij}\right)_{n\times n}\in {\rm GL}(n, \mathbb{Q})$.
%Roughly speaking, invertible polynomials are quasihomogeneous nondegenerate polynomials in which the numbers of variables and the numbers of monomials are equal.
%for admissible LG A-model pairs with an invertible polynomial and a maximal group.
%Here is a direct consequence of the mirror statements and Givental's Theorem.
\begin{theorem}
\label{thm-max}
Consider an admissible pair $(W, G_W)$ with an invertible polynomial $W$.
%Let $W$ be an invertible polynomial and $G_{W}$ be the group of diagonal symmetries of $W$.
\begin{enumerate}
\item
If $W$ has no weight-${1\over 2}$ chain variable, then
$L_k \mathcal{A}^{\clubsuit=\rm FJRW}_{W,G_W}=0$ for all $k\geq -1$.
\item
The equation $L_k \mathcal{A}^{\clubsuit=\rm PV}_{W,G_W}=0$ holds for all $k\geq -1$.
\end{enumerate}
\end{theorem}
%For invertible polynomials, we recall two mirror statements from the work \cite{HPSV, HLSW}, and obtain the Virasoro constraints below.
According to the mirror theorems in \cite{HLSW, HPSV}, the LG A-models in Theorem \ref{thm-max} are equivalent to Saito-Givental B-models of their mirror polynomials.
The mirror LG B models are generically semi-simple as the deformed polynomials are of Morse type in general.
%Thus  Theorem \ref{thm-max}  follows from the mirror statements and Givental Theorem \ref{semi-simple-Virasoro}.

Semisimple Frobenius manifolds also exists for admissible LG A-model pairs even if $G\neq G_W$.
For example, for any invertible polynomial of two variables, if we take $G$ to be the minimal group $\<J\>$.
Then the FJRW theory of $(W, \<J\>)$ may not be isomorphic to any admissible LG pair of invertible polynomial with a maximal group $G_W$.
We consider some examples of two-variable invertible polynomials in \cite[Section 4]{Francis}.
We compute  the quantum multiplication of the quantum Euler vector field for these examples.
Then by applying the criteria of semisimplicity descried in \cite{Abr}, we obtain the underlying Frobenius manifolds are generically semisimple.
As a consequence, we obtain
\begin{proposition}
\label{general-pillow}
Virasoro Conjecture \ref{VC} holds for all LG pairs $(W, \<J\>)$ if
$$W=x^4+y^4, x^3+y^6, x^3+y^9, x^4+y^6, x^3+xy^8.$$
\end{proposition}

\iffalse
\begin{remark}
{\red
Semisimple Frobenius manifolds also exists for admissible LG A-model pairs even if $G\neq G_W$.
For example, for any invertible polynomial of two variables, if we take $G$ to be the minimal group $\<J\>$.
Then the FJRW theory of $(W, \<J\>)$ may not be isomorphic to any admissible LG pair of invertible polynomial with a maximal group $G_W$.
In \cite{FHS}, we use the quantum multiplication of the quantum Euler vector field to prove that the underlying Frobenius manifolds for such LG pairs are also generically semisimple.
Thus
\begin{theorem}
\cite{FHS}
Virasoro Conjecture \ref{VC} holds for all LG pairs $(W, \<J\>)$ if $W$ is an invertible polynomial.
\end{theorem}
 We illustrate the strategy for the example of quartic polynomial
$$(W=x_1^4+x_2^4, \quad G=\<J_W\>).$$ The details will be appeared in \cite{FHS}.
Now we can cite the result when this example appear in the last section.}
\end{remark}
\fi
%In general, searching semi-simple Frobenius manifolds in quantum singularity theories is not trivial.
%It may require some work for chain type and loop type invertible polynomials. % with maximal group.

\subsubsection{LG pairs of Calabi-Yau type} % and Landau-Ginzburg/Calabi-Yau correspondence}
%Now we consider admissible LG pairs when
If a polynomial $W(x_1, \cdots, x_n)$ is of {\em Calabi-Yau type}, i.e., $\widehat{c}_W=n-2$,
%For such polynomials,
the hypersurface $(W=0)$ and the quotient space induced by the $G$-action are Calabi-Yau varieties.
In Witten's work of Gauge linear sigma models (GLSM) \cite{Wit}, the Gromov-Witten theory of the quotient Calabi-Yau variety has a deep connection to the LG A-model of the pair $(W, G)$. This is called the {\em Landau-Ginzburg/Calabi-Yau correpsondence}.
%In particular, the name {\em pillowcase} in Proposition \ref{prop-pillow} is borrowed via this correspondence.
%The CY counterpart of that LG model (after a stablization) is the elliptic orbifold curve $\mathbb{P}^{1}_{2,2,2,2}$, which is called the {\em pillowcase} in \cite{EO}.
%For more general LG pairs of CY type,
By comparing the FJRW theory with the Gromov-Witten theory of the Calabi-Yau counterpart, we obtain
\begin{theorem}
\label{thm-cy}
Let $W$ be an invertible polynomial of Calabi-Yau type.
For all $k\geq -1$, the Virasoro constraints $L_k \mathcal{A}^{\clubsuit=\rm FJRW}_{W,G}=0$ hold for the following admissible LG pairs $(W, G)$:
%Let $(W,G)$ be an admissible LG pair and $W$ be an invertible polynomial of Calabi-Yau type.
\begin{enumerate}
\item
$\widehat{c}_{W}\geq 3$ and $G\leq {\rm SL}(\C)$, under Assumption \ref{alge cycle};
\item
or %$\widehat{c}_{W}=1$,
$W$ is a Fermat CY  polynomial of three variables, and $G=\<J_W\>$. % then $L_k \mathcal{A}^{\clubsuit=\rm FJRW}_{W,\< J_W\>}=0$ for all $k\geq -1$.
\end{enumerate}
%Let $W$ be a invertible Calabi-Yau polynomial, that is $\widehat{c}_{W}=n-2$. If $\widehat{c}_{W}\neq 2$ and $G\leq {\rm SL}_n(\C)$, then the Virasoro Conjecture \ref{VC} holds for $D_{W,G}^{\rm FJRW}$. {\color{red}In fact, for $\widehat{c}_{W}\geq 3$, we don not need $W$ is CY, another condition is needed, see Theorem \ref{thm-lc}. Although $W$ is invertible CY and $G\leq {\rm SL}_n(\C)$  will imply that condition, see Corollary \ref{virasoro-high CY}.}
\end{theorem}
The Frobenius manifolds in these LG models are no longer generically semi-simple.
Part (1) of Theorem \ref{thm-cy} follows from a simple degree calculation, which is an analog of \cite[Theorem 7.1]{Ge}.
For part (2), we first recall the LG/CY correspondence proved in \cite{LLSZ, LSZ}.
The authors there relates the LG theory of the pair and the GW theory of the elliptic curve by a holomorphic Cayley transformation,
which is induced from the theory of quasi-modular forms.
The Virasoro operators in LG/CY theories commutes with the holomorphic Cayley transformation.
Thus Part (2) follows from the Virasoro constraints for the Gromov-Witten theory of the elliptic curve \cite{OP}. \\
%We may have to assume $W$ is Fermat here. But $G$ could be non-maximal.\\

%\newpage
\paragraph*{\bf Plan of the paper}
%\begin{itemize}
%\item
In Section \ref{sec-cohft}, we briefly review some properties of the CohFTs for the admissible LG pairs and discuss their influence on Virasoro constraints of $L_{-1}$ and $L_0$.
%\item
In Section \ref{sec-givental}, we follow Givental's work to express the Virasoro operators as quantization operators of certain quadratic Hamiltonians and prove the Virasoro relations \eqref{Virasoro-relation} in Proposition \ref{prop-virasoro}.
%\item
In Section \ref{sec-max}, we provide some examples of semisimple Frobenius manifolds.
%In Section \ref{sec-max}, we recall two mirror symmetry statements between admissible LG models of invertible polynomials with maximal group of symmetries and Givental-Teleman formula of semi-simple Frobenius manifolds.
Theorem \ref{thm-max} will follow from the mirror symmetry statements and Givental's proof of Virasoro constraints for semi-simple Frobenius manifolds.
We prove Proposition \ref{general-pillow} by calculating the quantum multiplication of the quantum Euler vector.
%\item
In Section \ref{sec-calabi-yau}, we discuss the connection between Virasro constraints and LG/CY correspondence, and prove Theorem \ref{thm-cy}.\\

%\end{itemize}

%In Section \ref{sec-hodge}, we discuss the connection between Virasoro operators and Hodge integrals in quantum singularity theories.

\iffalse
\paragraph*{\bf A list of notations}
\begin{itemize}
\item $\widehat{c}_W$: the central charge of $W$ in \eqref{central-charge}.
\item $G_W$: the group of diagonal symmetries of $W$ or the maximal group of $W$. % in \eqref{maximal-group}.
\item $J_W$ or $J$: the exponential grading element of $W$. % in \eqref{expo-element}.
\item $\cH_{W, G}$: the state space of $(W, G)$.
\item $(\mu_\phi^+,\mu_\phi^-)$: the bigrading of an element $\phi=\alpha|\gamma\rangle\in \cH_{W, G}$.
\item $|\phi|$: the parity of $\phi$.
\item $\iota_{\gamma}$: degree shifting number of $\gamma\in G$.
\item $N_\gamma$: dimension of the fixed locus of $\gamma$.
\item $(\phi_a, \phi_b)$: the pairing between $\phi_a$ and $\phi_b$.\\
\end{itemize}
\fi

\paragraph*{\bf Acknowledgement}
We would like to thank Huijun Fan, Todor Milanov, Alexander Polishchuk, Yongbin Ruan, and Arkady Vaintrob for helpful discussions.
W. He would like to thank Jianxun Hu, Huazhong Ke, Xiaowen Hu, Yifan Li, Xiaobo Liu and Xin Wang.
Y. Shen would like to thank Amanda Francis, Takashi Kimura, Y-P Lee, Hsian-Hua Tseng, and Jie Zhou.

We thank the hospitality of Institute for Advanced Study in Mathematics, Zhejiang University. Part of the work was done there during the authors' visit in September 2019.
Y. Shen is partially supported by Simons Collaboration Grant 587119.

%\newpage

\section{%Properties of CohFT and the
Virasoro constraints of $L_{-1}$ and $L_0$}
\label{sec-cohft}
In this section, we discuss some properties of the cohomology field theories and the connections to Virasoro operators $L_{-1}$ and $L_0.$

\subsection{String equation and the operator $L_{-1}$}\label{string}
%We will mainly consider the FJRW theory.
%We notice that the definition of ancestor invariants differ from the FJRW invariants in \cite[Definition 4.2.1]{FJR2} by a factor of $|G|^g$. These factors may play some roles in the differential equations.
According to \cite[Theorem 4.2.2]{FJR}, \cite[Theorem 5.1.2]{PV}, and \cite[Theorem 4.6]{KL}, each of the three CohFTs $\{\Lambda^{\clubsuit, W, G}_{g,k}\}$ has an element
$$\one:=\phi_0:=1|J\>,$$
called a {\em flat identity}, which satisfies the following two conditions:
\begin{itemize}
\item
The CohFT is compatible with the pairing
$$\int_{\overline{\cM}_{0,3}}\Lambda_{0,3}^{\clubsuit, W, G}(\phi_a,\phi_b,\one)=(\phi_a, \phi_b).$$
\item
If $2g-2+k>0$, let $p: \overline{\cM}_{g,k+1}\to \overline{\cM}_{g,k}$ be the morphism forgetting the last marking and contracting all the unstable components, then
$$\Lambda_{g,k+1}^{\clubsuit, W, G}(\phi_1, \cdots, \phi_k, \one)=p^*\Lambda_{g,k}^{\clubsuit, W, G}(\phi_1, \cdots, \phi_k).$$
\end{itemize}
Using the geometry of psi-classes, these two conditions imply:
%we recall that the string equation in FJRW theory \cite[Theorem 4.2.9]{FJR} (The string equations also hold in PV theory and KL theory.)
%We notice that these equations are for ancestor potential functions. We also recall a grading equation. These equations are equivalent to the Virasoro constraints $L_{k}Z=0$ for $k\leq 0$, except the cases when $\widehat{c}_W=3$.
% the string equations are
\begin{align}
\LD\tau_{0}(\phi_a)\tau_{0}(\phi_b)\tau_{0}(\one)\RD_{0,3}^{\clubsuit, W, G}&=\eta_{ab}\label{string-special};\\
\LD\prod_{i=1}^{k}\tau_{\ell_i}(\phi_{i})\tau_{0}(\one)\RD_{g,k+1}^{\clubsuit, W, G}
&=\sum\limits_{j=1}^{k}\LD\prod_{i=1}^{k}\tau_{\ell_i-\delta_{i}^{j}}(\phi_{i})\RD_{g,k}^{\clubsuit, W, G}.
%\<\tau_{\ell_1}(a_1), \cdots, \tau_{\ell_i-1}(a_i), \cdots\>_{g,k}.
\label{string}
\end{align}
These equations are called {\em string equations}.
The Virasoro operator $L_{-1}$ in \eqref{virasoro-operator} is given by
$$L_{-1}=-{\partial\over \partial t_{0}^{0}}+\sum_{a, m} t^a_{m+1}{\partial\over \partial t_{m}^{a}}+{1\over 2\hbar^2}\sum_{a,b} \eta_{ba}t^{a}_{0}t^{b}_0.$$
The peculiar ordering of the variables here reflects the potential presence of odd classes.
Immediately, we see the string equations \eqref{string} and \eqref{string-special} is equivalent to the Virasoro constraint
$$L_{-1}\mathcal{A}^{\clubsuit}_{W,G}=0.$$

\subsection{A grading equation and the operator $L_0$}
\label{A grading equation and the operator $L_0$}

Now we consider the Virasoro constraint of $L_0$.
We need some preparations.
%The dilaton equation and the grading equation
Recall the parity $|\phi_a|$ defined in Definition \ref{parity}, we define the {\em Euler characteristic} of the admissible LG pair $(W, G)$ to be
\begin{equation}
\label{euler-char}
\chi_{W,G}:=\sum_{\phi_a\in \cH_{W,G}}|\phi_a|.
\end{equation}
\begin{remark}
In FJRW theory, the parity is induced by the cohomological degree of Lefschetz thimbles; in PV theory, the parity is induced by the Hochschild degree of Hochschild homology class.
\end{remark}

Next we define a {\em Hodge grading operator} $\theta$ by
\begin{equation}
\label{hodge-op}
\theta(\phi_az^m):=\mu^+_a\cdot \phi_az^m.
\end{equation}
Now we rewrite the last term in the Virasoro operator $L_k$ in \eqref{virasoro-operator} using a {\em supertrace}
\begin{equation}
\label{super-trace}
%:=\sum_{a}(-1)^{|\phi_a|}(\mu_a^+-{1\over 2})(\mu_a^++{1\over 2})
{\rm Str}(\theta^2-{1\over 4}):=\sum_{a}(-1)^{|\phi_a|}(\mu_a^+-{1\over 2})(\mu_a^++{1\over 2}).
%={\rm Str}(\theta^2)-{\chi_{W,G}\over 4}.
\end{equation}
%\subsubsection{A supertrace formula}
Similar to \cite[Proposition 2.6]{Ge}, the super trace term ${\rm Str}(\theta^2-{1\over 4})$ is related to the {\em Euler characteristic} $\chi_{W,G}$.
%${\rm Str}(\theta^2)={\chi\cdot\widehat{c}\over 12}$, which implies
\begin{proposition}
[A supertrace formula]
\label{conj-supertrace}
Let $(W, G)$ be an admissible LG pair, then
\begin{equation}
\label{hodge-rr}
{\rm Str}(\theta^2-{1\over 4})%:=\sum_{a}(-1)^{|\phi_a|}(\mu_a^+-{1\over 2})(\mu_a^++{1\over 2})
={\widehat{c}_W-3\over 12}\chi_{W,G}.
\end{equation}
\end{proposition}
This formula will be discussed in Section \ref{sec-super}.
As a consequence, the Virasoro operator $L_0$ has the following expression:
\begin{align}%\label{L_0-operator}
L_0:=-\frac{3-\widehat{c}_W}{2} {\pop t_{1}^{0}}
+\sum_{m=0}^\infty\left(\mu^+_a+m+\frac{1}{2}\right) t^a_m{\pop t_{m}^{a}}-{\widehat{c}_W-3\over 48}\chi_{W,G}\label{Virasoro-L0}.
\end{align}
%In order to deal with the other two terms in \eqref{Virasoro-L0}, we need to introduce the dilaton equation and a grading equation.

\subsubsection{Dilaton equations}
The {\em dilaton equations} are:
\begin{align}
\LD\tau_{1}(\one)\RD_{1,1}^{\clubsuit, W, G}&={\chi_{W,G}\over 24}.\label{dilaton-g=1}\\
\LD\prod_{i=1}^{k}\tau_{\ell_i}(\phi_{i})\tau_{1}(\one)\RD_{g,k+1}^{\clubsuit, W, G}&=(2g-2+k)\LD \prod_{i=1}^{k}\tau_{\ell_i}(\phi_{i})\RD_{g,k}^{\clubsuit, W, G}.\label{dilaton}
\end{align}
These equations are obtained by using the geometry of psi-classes and the virtual fundamental cycles.
For example, the first equation can be deduced from the tautological relation
$$\psi_1={1\over 24}\cdot\delta_{\rm irr}\in H^2(\overline{\cM}_{1,1}),$$
where $\delta_{\rm irr}$ is the boundary divisor of $\overline{\cM}_{1,1}$, parametrizing all stable genus one curves with one node and one marking.

Similar to the string equations, the dilaton equations \eqref{dilaton-g=1} and \eqref{dilaton} can be rewritten as differential equations
\begin{equation}
\label{dilaton-operator}
\left(-{\partial\over \partial t_{1}^{0}}+\sum_{a, m} t^a_{m}{\partial\over \partial t_{m}^{a}}+\hbar{\partial\over \partial \hbar}+{\chi_{W,G}\over 24}\right)\mathcal{A}^{\clubsuit, W, G}=0.
\end{equation}
Again, the super trace formula \eqref{hodge-rr} allows us to replace the scalar multiplication by the Euler characteristic $\chi_{W,G}$  in \eqref{dilaton-operator} by a term related to the super trace ${\rm Str}(\theta^2-{1\over 4})$.
%The second last operator  $\hbar{\partial\over \partial \hbar}$ in \eqref{dilaton-operator} is a genus shifting operator.

%When $\widehat{c}_{W}\neq 3$, we can eliminate such a term using a grading equation.

\subsubsection{A grading equation}
We define a grading operator
\begin{equation}\label{bigE}
	 \widetilde{E}:=\sum_{m, a}\left(m-1+\mu^+_a+{\widehat{c}_W\over{2}}\right)t_m^{a}{\partial \over\partial t_m^{a}}.
\end{equation}
\begin{proposition}
 %If the total ancestor potential $\mathcal{A}^{\clubsuit}_{W,G}$ satisfies
The Virasoro constraint $L_0\mathcal{A}^{\clubsuit}_{W,G}=0$ is equivalent to the {\em grading equation}
\begin{equation}
\label{grading-equation}
\widetilde{E}\mathcal{A}^{\clubsuit}_{W,G}=\left({3-\widehat{c}_W\over 2}\right)\hbar{\partial\over \partial \hbar}\mathcal{A}^{\clubsuit}_{W,G}.
\end{equation}
\end{proposition}
\begin{proof}
There are two situations.
If $\widehat{c}_W=3$, %as a consequence of the super trace formula \eqref{hodge-rr},
we compare \eqref{Virasoro-L0} and \eqref{bigE}.
%${\rm Str}(\theta^2-{1\over 4})=0.$
The grading equation $\widetilde{E}(\mathcal{A}^{\clubsuit}_{W,G})=0$ is exactly the Virasoro constraint
$L_0\mathcal{A}^{\clubsuit}_{W,G}=0.$

If $\widehat{c}_W\neq 3$, then we can cancel the differential operator $\hbar{\partial\over \partial \hbar}$ by taking a linear combination of the dilaton equation \eqref{dilaton-operator} and the grading equation \eqref{grading-equation}, and get the Virasoro constraint
%$$\left(-{\partial\over \partial t_{1}^{0}}+\sum t^a_{k}{\partial\over \partial t_{k}^{a}}+{2\over 3-\widehat{c}_W}\sum_{k, a}(k-1+\mu^+_a+{\widehat{c}_W\over{2}})t_k^{a}{\partial \over\partial_k^{a}}+{\chi\over 24}\right)Z=0.$$ which reads,
\begin{equation}\label{L0}
 \left({\widehat{c}_W-3\over 2}{\partial\over \partial t_{1}^{0}}+\sum_{m, a}(m+\mu^+_a+{1\over 2})t_m^{a}{\partial \over\partial t_m^{a}}+{(3-\widehat{c}_W)\chi\over 48}\right)\mathcal{A}^{\clubsuit}_{W,G}=0.
%\left(-{3-\widehat{c}_W\over 2}\cdot {\partial\over \partial t_{1}^{0}}+\sum_{k, a}\left({3\over 2}+k-1+\mu^+_a\right)t_k^{a}{\partial \over\partial_k^{a}}+{3-\widehat{c}_W\over 2}\cdot {\chi\over 24}\right)\mathcal{A}^{\clubsuit}_{W,G}=0.
\end{equation}
%which is exactly the Virasoro constraint $L_0\mathcal{A}^{\clubsuit}_{W,G}=0$ by \eqref{Virasoro-L0}.
\end{proof}

\subsubsection{A sufficient condition}
%{Algebraic virtual cycles}
We provide a sufficient condition for the grading equation \eqref{grading-equation}.
%We make the following assumption.
 \begin{assumption}\label{alge cycle}
   The quantum invariant
   $$\LD\prod_{i=1}^{k}\tau_{\ell_i}(\phi_{i})\RD^{\clubsuit, W, G}_{g, k}\neq 0$$
   only if
   $$\sum_{i=1}^k\left(\mu_{i}^{\pm}+{\widehat{c}_W\over 2}+\ell_i\right)=(3-\widehat{c}_W)(g-1)+k.
   $$
 \end{assumption}
One can check directly that for all the examples in Theorem \ref{thm-max}, Theorem \ref{general-pillow}, and Theorem \ref{thm-cy}, Assumption \ref{alge cycle} holds true.
In fact, using the bigrading in \eqref{bigrading}, we obtain  %Now combining \eqref{L0} and Proposition \ref{conj-supertrace}, we obtain
\begin{proposition}
If Assumption \ref{alge cycle} is satisfied for $\{\Lambda^{\clubsuit, W, G}_{g,k}\}$ of an admissible LG pair $(W, G)$, then
 $L_0\mathcal{A}^{\clubsuit}_{W,G}=0.$
\end{proposition}
% {\red Check if it uses the homogeneity condition \eqref{deg-fjrw}.}

\begin{remark}
The FJRW invariants satisfies a degree constraint \cite[Theorem 4.1.8 (1)]{FJR}
\begin{equation}
\label{deg-fjrw}
\deg_\C[\cW_{g, \gamma}^G]^{\rm{vir, FJRW}}=(3-\widehat{c}_W)(g-1)+n-\sum_{i=1}^{n}\iota_{\gamma_i}.
%=(3-\widehat{c}_W)(g-1)+n-\sum_{i=1}^{k}\deg\phi_i
\end{equation}
%Here $\iota_{\gamma_i}$ is the degree-shift number defined in Definition \ref{bigrad}.
The KL theory satisfies the same constraint, which is induced by \cite[(3.20)]{CKL}.
Such a constraint is conjectural in PV theory, called the {\em Homogeneity Conjecture} \cite[Section 5.6]{PV}.
We see Assumption \ref{alge cycle} is stronger than \eqref{deg-fjrw}.
\end{remark}
%Now Theorem \ref{thm-low} is a consequence of the following result.
%\begin{proposition}
%This assumption \ref{alge cycle} holds in KL theory.
%\end{proposition}

\iffalse
\begin{remark}
[Internal grading and external grading]\label{mixed Hodge}
The element $\alpha|\gamma\>$ has an {\em internal grading} $\big({\rm wt}(\alpha), N_\gamma-{\rm wt}(\alpha)\big)$ and an {\em external grading} $(\iota_\gamma, \iota_\gamma).$
The internal A-model bigrading has following interpretation. The FJRW (broad) state space is the space of the relative cohomology $H^n(\C^n, W^{\infty}, \C)$.
It has a mixed Hodge structure which defines Hodge grading and Hodge decomposition
	$$H^n(\C^n, W^{\infty}, \C)=\bigoplus_p H^{p, n-p}.$$
Under the isomorphism to the $\Jac(W)d\bx$, the component $H^{p, n-p}$ corresponds to {\red the degree $(p, n-p)$-component of $\Jac(W)d\bx$}. So the internal bigrading is just the Hodge grading of the Lefschetz thimble. See \cite{FJR}, \cite{Ku}.
% The bigrading in KL theory is discussed in Remark \ref{KL bigrad}
%When we consider KL theory, the (broad) state space is the intersection homology $IH_n(W^{-1}(0))$, which also admits a mixed Hodge structure \cite{Sai}. The Hodge bigrading also coincides with the bigrading given above \cite{Zu}.
%\textcolor{red}{What about B-model?} % internal/external grading?}
\end{remark}
\fi

\begin{remark}\label{KL bigrad}
A similar condition as Assumption \ref{alge cycle} exists in Gromov-Witten theory of smooth algebraic varieties.
It is closely related to the motive axiom of cohomological field theory \cite{KM}. The motive axiom states that Gromov-Witten theory is induce by some algebraic cycle via Fourier-Mukai transform, which is proved by Li-Tian \cite{LiT} and Behrend-Fantechi \cite{B, BF}. More explicitly, for $V$ is  a smooth algebraic variety, the GW CoFT
$$\Lambda^{\rm GW}_{g, n, \beta}: H^*(V)^{\otimes n}\rightarrow H^*(\overline{\mathcal{M}}_{g, n})$$
is induced by some cocycle $C_{g, n, \beta}$ in the Chow ring $A^*(V^n\times \overline{\mathcal{M}}_{g, n})$ via the following Fourier-Mukai transform:
\begin{equation}\label{FM transform}
\Lambda^{\rm GW}_{g, n, \beta}(\phi_1\otimes \cdots \otimes \phi_n)=(pr_{2})_*[pr_1^*(\phi_1\wedge \cdots \wedge \phi_n)\wedge C_{g, n, \beta}],
\end{equation}
where $pr_1: V^n\times \overline{\mathcal{M}}_{g, n}\rightarrow V^n$, $pr_2: V^n\times \overline{\mathcal{M}}_{g, n}\rightarrow \overline{\mathcal{M}}_{g, n}$ are the projections.
Now consider the Hodge decomposition $H^*(V^n\times \overline{\mathcal{M}}_{g, n})=\bigoplus_{p, q}H^{p, q}$. Since $C_{g, n, \beta}$ is an algebraic cycle, we know $C_{g, n, \beta}\in H^{d, d}$ , where $2d$ is the cohomological degree of  $C_{g, n, \beta}$. Furthermore, GW invariant is obtained via the following integration:
$$\<\phi_1, \cdots , \phi_n\>^{\rm GW}_{g, n, \beta}=\int_{\overline{\mathcal{M}}_{g, n}}\Lambda^{\rm GW}_{g, n, \beta}(\phi_1\otimes \cdots \otimes \phi_n).$$
The invariant does not vanish means that
$$\Lambda^{\rm GW}_{g, n, \beta}(\phi_1\otimes \cdots \otimes \phi_n)\in H^{3g-3+n, 3g-3+n}(\overline{\mathcal{M}}_{g, n}).$$ Suppose $\phi_i\in H^{p_i, q_i}(V)$, then \eqref{FM transform} and the.above argument will induce the following constrains:
$$\sum_{i=1}^n p_i=\sum_{i=1}^n q_i.$$
Assumption \ref{alge cycle} in GW theory is a consequence of the above equality and dimension axiom of GW theory.

\end{remark}

%{\red List some other examples when Assumption \ref{alge cycle} hold. For example, $G_W$ as no odd parity. We don't require $W$ to be invertible here.}

%%%%%%Theorem
%\begin{corollary}
%If $W$ is an invertible polynomial and $G$ is an admissible group, then the following Virasoro constraint holds:
%$$L_{0}\mathcal{D}_{W,G}=0.$$
%\end{corollary}

%\subsection{Genus zero and Genus one} {\red Now discuss the consequence from \cite{LT} \cite{DZ} for quantum singularity theories.}

\section{Virasoro relations and Givental formalism}
\label{sec-givental}
In this section, we recall Givental's work on quantization of quadratic Hamiltonians \cite{G-ham}.
It implies the Virasoro relation \eqref{Virasoro-relation} in Proposition \ref{prop-virasoro}.
\subsection{Quantization of quadratic Hamiltonians}
This section is mainly based on \cite{G-ham}. See also \cite{CPS} for an exposition.

Fix a $\Z_2$-graded state space $\cH$, set
$$\mathbb{H}:=\cH(\!(z^{-1})\!).$$
A choice of basis $\{\phi_a\}$ for $\cH$ yields a symplectic basis for $\mathbb{H}$, in which the expression for an arbitrary element in Darboux coordinates is
\begin{equation}
\label{darboux-element}
f(z)=\sum_{a}\sum_{\ell\geq 0}p_{\ell, a}\phi^a(-z)^{-\ell-1}+\sum_{b}\sum_{m\geq 0}q_m^b\phi_b (z)^m\in \mathbb{H}.
\end{equation}
%$$\Phi=\sum_{k\geq 0}p_{k, \alpha}\phi^\alpha(-z)^{-k-1}+\sum_{l\geq 0}q_l^\beta\phi_\beta z^l,$$
Here the parity of $p_{\ell, a}$, $q_m^b$ coincides that of $\phi^a$ and $\phi_b$ respectively.
The pairing $( , )$ on $\cH$ induces a symplectic form on $\mathbb{H}$, denoted by
$$\Omega(f(z), g(z)):={\rm{Res}}_{z=0}\big(f(-z), g(z)\big).$$

\begin{definition}
An operator $A: \mathbb{H}\rightarrow \mathbb{H}$ is called an {\em infinitesimal symplectic transformation} if
$$\Omega(Af, g)+\Omega(f, Ag)=0.$$
An operator $T: \mathbb{H}\rightarrow \mathbb{H}$ is called a {\em symplectic transformation} if
$$\Omega(Tf, Tg)=\Omega(f, g).$$
\end{definition}

Now define the quantization of quadratic terms by % the {\em canonical commutation relations}
\begin{equation}\label{ccr}
\begin{dcases}
\widehat{q_iq_j}&:={q_iq_j\over \hbar},\\
\widehat{p_ip_j}&:=\hbar\partial_{q_i}\partial_{q_j},\\
\widehat{q_ip_j}&:=q_i\partial_{q_j}.
\end{dcases}
\end{equation}
The quantization of quadratic Hamitonians forms a projective representation of the Poisson (Lie) algebra.

For an infinitesimal symplectic transformation $A$, define the quadratic Hamiltonian by
\begin{equation}
h_{A}:={1\over2}\Omega(A\Phi, \Phi),
\end{equation}
and the quantization of $A$ is defined to be the quantization of $h_A$ via \eqref{ccr}:
\begin{equation}\label{infinitesimal-symp}
	\widehat{A}:=\widehat{h_A}.
\end{equation}
For a symplectic transformation $T=\exp A$, define its quantization as
\begin{equation}
\widehat{T}:=\exp \widehat{A}.
\end{equation}

\begin{lemma}\label{key-cocycle}
\cite{G-ham}
	Let $A_1, A_2$ be two infinitesimal symplectic transformations, then
	$$[\widehat{A_1}, \widehat{A_2}]=[A_1, A_2]^{\wedge}+C(h_{A_1}, h_{A_2}),$$
	where $C(,)$ is a cocycle defined by
	\begin{align}\label{cocycle}
		C(p_ap_b, q_aq_b)&=(-1)^{|p_b||q_a|}+\delta_{a,b},
	\end{align}
and $C=0$ on any other pair of quadratic Darboux monomials.
		 Here cocycle means $C$ is a closed 2-form in the Hochschild cochain of the Lie algebra of quadratic Hamiltonian.
	\end{lemma}
		 For simplicity, we denote
		 $$C(A_1, A_2):=C(h_{A_1}, h_{A_2}).$$
\begin{corollary}
\label{cocycle-commute}
Let $A$ be an infinitesimal symplectic transformation and $T$ be a symplectic transformation, we have
       \begin{equation}
       \label{S-commute}
       \widehat{T}\circ\widehat{A}\circ\widehat{T}^{-1}=\widehat{TAT^{-1}}+C_T(A).
       \end{equation}
       Here $C_T(A)$ is the constant defined to be
       $$C_T(A)=C\left(\log T, \sum_{n=0}\frac{1}{(n+1)!}{\rm Ad}_{\log T}^n(A)\right).$$
	\end{corollary}
%	The first statement is by direct checking via definition.
\begin{proof}
Let $T=\exp{B}$, then \eqref{S-commute} is just
	\begin{align*}
		\widehat{T}\circ\widehat{A}\circ\widehat{T}^{-1}
		&=\exp(\mathrm{Ad}_{\widehat{B}})(\widehat{A})\\
		&=\widehat{\exp(\mathrm{Ad}_{B})(A)}+C\left(B, \sum_{n=0}\frac{1}{(n+1)!}{\rm Ad}_{B}^n(A)\right),
	\end{align*}
	where $\mathrm{Ad}_{B}(\widehat{A}):=[B, A]$	.
\end{proof}

In this paper, we mainly consider operators on $\mathbb{H}$ of following two types, {\em lower triangular operators} of the form
\begin{equation}
\label{upper-triangular}
	S(z)=S_0+S_1z^{-1}+S_2z^{-2}+\cdots, \quad S_i\in {\rm End}(\cH),
\end{equation}
and {\em upper triangular operators} of the form
\begin{equation*}
	R(z)=R_0+R_1z+R_2z^2+\cdots, \quad R_i\in {\rm End}(\cH).
\end{equation*}
When we consider pair or Lie bracket between positive sum and negative sum, we always assume one of them is finite, such that the pair or Lie bracket is well-defined.

\subsection{Virasoro relations}
\label{sec-relation}
Now we return to an admissible LG pair $(W, G)$.
Recall that $(\mu_a^+, \mu_a^-)$ is the bigrading of the element $\phi_a=[f_a]d{\bf x}_{\gamma_a}\in \cH_{W, G}$ defined in \eqref{bigrading}.
Using \eqref{bigrading} and \eqref{qst-pairing}, it is easy to check that
\begin{lemma}\label{pairing}
We fix  a basis of $\cH_{W,G}$ as follows $\left\{\phi_a=f_a|\gamma_a\>\mid [f_a]\in {\rm Jac}(W_{\gamma_a})\right\}.$
If the pairing
$\eta_{ab}:=(\phi_a, \phi_b)\neq 0,$ then
$\mu^+_a+\mu^+_{b}=0.$
\end{lemma}
The Hodge grading operator $\theta$ defined in \eqref{hodge-op} is an operator on $\mathbb{H}$.
We introduce an auxiliary operator
\begin{equation}
\label{auxiliary-operator}
D:=z(\partial_z+z^{-1}\theta)z=z^2\partial_z+z(\theta+1).
\end{equation}
It satisfies
\begin{equation}
\label{d-commute-z}
[D, z^{-1}]=-1.
\end{equation}
We will consider a sequence of differential operators
\begin{equation}
\label{virasoro-diff}
\cL_k:=z^{-1/2}D^{k+1}z^{-1/2}, \quad k\geq -1.
\end{equation}
Using Lemma \ref{pairing} and \eqref{d-commute-z}, we obtain
\begin{lemma}
For all integers $k\geq -1$, $\cL_k$ is infinitesimal symplectic, and
\begin{equation}\label{a}
[\cL_m, \cL_n]=(m-n)\cL_{m+n}.
\end{equation}
\end{lemma}
Recall $\eta^{ab}$ is the $(a, b)$-th entry of the inverse matrix of $(\eta_{ab})$ and $f(z)$ is the Darbourx coordinate defined in \eqref{darboux-element}.
\begin{lemma}\label{quadratic-Lk}
The differential operator $\cL_k$ induces a quadratic Hamiltonian
\begin{align}
\Omega(\cL_k f, f)=&-\delta_{k,-1}\sum_{a,b}q_0^aq_0^b\eta_{ab}\nonumber\\
&-2\sum_{a}\sum_{\substack{m\geq 0\\m\geq-k}}\prod_{j=1}^{k+1}(\mu_{a}^++m-{1\over 2}+j)p_{m+k,a}q_m^a\label{pq-term}\\
&+\sum_{a, b}\sum_{\ell=0}^{k-1}(-1)^{\ell}\prod_{i=1}^{k+1}(\mu_a^+-\ell-{3\over 2}+i)p_{k-\ell-1,b}p_{\ell, a}\eta^{b,a}.\label{pp-term}
\end{align}
\end{lemma}
\begin{proof}
%Now $$f(-z)=\sum_{\ell\geq 0}p_{\ell, \alpha}\phi^\alpha(z)^{-\ell-1}+\sum_{m\geq 0}q_m^\beta\phi_\beta (-z)^m$$
The proof follows from direction calculations.
Using \eqref{virasoro-diff}, we have
\begin{eqnarray*}
\cL_{k} f(z)
&=&\sum_{a}\sum_{\ell\geq 0}\prod_{i=1}^{k+1}(\mu_a^+-\ell-{3\over 2}+i)p_{\ell, a}\phi^a (-z)^{-\ell-1}z^{k}\\
&&+\sum_{b}\sum_{m\geq 0}\prod_{j=1}^{k+1}(\mu_{b}^++m-{1\over 2}+j)q_m^b\phi_b z^{m+k}.
\end{eqnarray*}
Here if $k=-1$, we have the convention
$$\prod_{i=1}^{k+1}(\mu_a^+-\ell-{3\over 2}+i)=1.$$

%Recall that $\{\phi^b\}$ is the dual basis of $\{\phi_a\}$. So we have $$(\phi_a, \phi^b)=\delta_a^{b}; \quad (\phi_a, \phi_b)=\eta_{a,b}.$$
%Also by Lemma  \ref{pairing}, if $(\phi_a, \phi_b)\neq0$, then $\mu_a^++\mu_b^+=0$.

The quadratic Hamiltonian $\Omega(\cL_k f, f)$ contains at most three types of terms: $q_iq_j$-term, $q_ip_j$-term, and $p_ip_j$-term. We discuss each type in details using the formula
$$\Omega(\cL_k f, f)=-\Omega(f, \cL_k f).$$
\begin{enumerate}
\item
The $q_iq_j$-terms appear in $-\Omega(f, \cL_k f)$ only if $k=-1$. They are
$$-{\rm Res}_{z=0}\bigg(q_0^a\phi_a, q_0^b\phi_b z^{-1}\bigg)=-\sum_{a,b}q_0^aq_0^b\eta_{ab}.$$
\item
The $q_ip_j$-terms appear in two situations. The term appears in \eqref{pq-term} is a combination of the contributions from the two situations.
\begin{itemize}
\item In the first situation, we have
\begin{align*}
&-\bigg(\sum_{a}\sum_{\ell\geq 0}p_{\ell, a}\phi^a(z)^{-\ell-1}, \sum_{m\geq 0}\sum_{b}\prod_{j=1}^{k+1}(\mu_{b}^++m-{1\over 2}+j)q_m^b\phi_b z^mz^k\bigg)\\
=&-\sum_{a,b}\sum_{\substack{m\geq 0\\m\geq-k}}\prod_{j=1}^{k+1}(\mu_{b}^++m-{1\over 2}+j)p_{m+k,a}q_m^b(\phi^a,\phi_b).
\end{align*}
\item In the second situation, we have
\begin{align*}
&-\bigg(\sum_{b}\sum_{m\geq 0}q_m^b\phi_b (-z)^m, \sum_{\ell\geq 0}\sum_{a}\prod_{i=1}^{k+1}(\mu_a^+-\ell-{3\over 2}+i)p_{\ell, a}\phi^a (-z)^{-\ell-1}z^{k}\bigg)\\
=&\sum_{a,b}\sum_{\substack{m\geq 0\\ m\geq -k}}(-1)^{k}\prod_{i=1}^{k+1}(\mu_a^+-\ell-{3\over 2}+i)q_m^bp_{m+k,a}(\phi_b, \phi^a).
\end{align*}
\end{itemize}
Recall in Lemma \ref{pairing},  we have $(\phi_a, \phi^b)=\delta_a^{b}$. Now taking $\ell=m+k$, we obtain
\begin{align*}
\prod_{j=1}^{k+1}(\mu_{b}^++m-{1\over 2}+j)
&=\prod_{j=1}^{k+1}(-\mu_{a}^++m-{1\over 2}+j)\\
&=(-1)^{k+1}\prod_{j=1}^{k+1}(\mu_a^+-\ell+k+{1\over 2}-j).
\end{align*}
So the terms from two situations are equal.
So we obtain the total contribution is the term in \eqref{pq-term}.
%So the total contribution is $$-2\sum_{\substack{m\geq 0\\m\geq-k}}\prod_{j=1}^{k+1}(\mu_{b}^++m-{1\over 2}+j)p_{m+k,b}q_m^b.$$
\item
Finally, the $p_ip_j$-terms are (by taking $m=k-\ell-1$)
\begin{align*}
&-\bigg(\sum_{b}\sum_{m\geq 0}p_{m, b}\phi^b(z)^{-m-1}, \sum_{a}\sum_{\ell\geq 0}\prod_{i=1}^{k+1}(\mu_a^+-\ell-{3\over 2}+i)p_{\ell,a}\phi^a (-z)^{-\ell-1}z^{k}\bigg)\\
=&\sum_{a,b}\sum_{\ell=0}^{k-1}(-1)^{\ell}\prod_{i=1}^{k+1}(\mu_a^+-\ell-{3\over 2}+i)p_{k-\ell-1}^{b}p_{\ell,a}\eta^{ba}.
\end{align*}
This is the contribution in \eqref{pp-term}.
\end{enumerate}
\end{proof}

\begin{definition}
Let $\widehat{\cL_k}$ be the quantization operator of the differential operator $\cL_k$ in \eqref{virasoro-diff}, defined by the formulas \eqref{infinitesimal-symp} and \eqref{ccr}.
\end{definition}
Using Lemma \ref{quadratic-Lk},
we can calculate these quantization operators explicitly.
\begin{lemma}
For $k\geq -1$, the quantization operator $\widehat{\cL_k}$ has the form
\begin{eqnarray*}
\widehat{\cL_k}(\mathbf{q})&=&-{\delta_{k,-1}\over 2\hbar}\sum_{a,b}q_0^aq_0^b\eta_{a,b}\nonumber\\
&&-\sum_{a}\sum_{\substack{m\geq 0\\m\geq-k}}\prod_{j=1}^{k+1}(\mu_{a}^++m-{1\over 2}+j)q_m^a {\pop q_{m+k}^{a}}\nonumber\\
&&+{\hbar\over 2}\sum_{a,b}\sum_{\ell=0}^{k-1}(-1)^{\ell}\prod_{i=1}^{k+1}(\mu_a^+-\ell-{3\over 2}+i) {\pop q_{k-\ell-1}^{b}} {\pop q_{\ell}^{a}}\eta^{ba}.
\end{eqnarray*}
\end{lemma}
%\newpage

Now we can relate the  quantization operators $\{\widehat{\cL_k}\}$ to the differential operators $\{L_k\}$ defined in \eqref{virasoro-operator} by
considering the {\em dilaton shift}
$$q^a_m:=t^a_m-\delta_{m, 1}\delta_{a, 0}.$$
\begin{proposition}\label{Virasoro-quantization}
We have
\begin{equation}
\label{virasoro-equality}
L_k=\widehat{\cL_k}-\delta_{k,0}\cdot{1\over4}\sum_{a}(\mu_a^+-{1\over 2})(\mu_a^++{1\over 2}), \quad k\geq -1.
\end{equation}
\end{proposition}
\begin{proof}
Since $\mu_{a=0}^+=-\widehat{c}_W/2$,
using the dilaton shift $q_1^0=t_1^0-1$,
we have
 \begin{equation} \label{coeff-dilaton}
-(-1)\prod_{j=1}^{k+1}(\mu_{a=0}^++1-{1\over 2}+j)p_{1+k,a=0}%=\prod_{j=1}^{k+1}(\mu_{a=0}^++1-{1\over 2}+j)\partial_{k+1,0}
=\prod_{j=1}^{k+1}(-{\widehat{c}_W-1\over 2}+j){\pop t_{k+1}^{0}}.\end{equation}
Thus
\begin{align*}
\widehat{\cL_k}(\mathbf{t})
=&\prod_{j=1}^{k+1}(-{\widehat{c}_W-1\over 2}+j){\pop t_{k+1}^{0}}\\
&-{\delta_{k,-1}\over 2\hbar}\sum_{a,b}t_0^a t_0^b\eta_{a,b}\nonumber\\
&-\sum_{a}\sum_{\substack{m\geq 0\\m\geq-k}}\prod_{j=1}^{k+1}(\mu_{a}^++m-{1\over 2}+j)t_m^a{\partial\over \partial t_{m+k}^{a}}\nonumber\\
&+{\hbar\over 2}\sum_{a,b}\sum_{\ell=0}^{k-1}(-1)^{\ell}\prod_{i=1}^{k+1}(\mu_a^+-\ell-{3\over 2}+i){\partial\over \partial t_{k-\ell-1}^{b}}{\partial\over \partial t_{\ell}^{a}}\eta^{ba}\\
=&L_k(\mathbf{t})+\delta_{k,0}\cdot{1\over4}\sum_{a}(\mu_a^+-{1\over 2})(\mu_a^++{1\over 2}).
\end{align*}
	\end{proof}

{\em A proof of Proposition \ref{Virasoro-relation}.}
%Now we give a proof of Proposition \ref{Virasoro-relation},
If $\{m, n\}\neq \{1, -1\}$, by  Lemma \ref{key-cocycle} and \eqref{a}, we have
\begin{align*}
	[L_m, L_n]=[\widehat{\cL_m}, \widehat{\cL_n}]=\widehat{[\cL_m, \cL_n]}=(m-n)\widehat{\cL_{m+n}}=(m-n)L_{m+n}.
\end{align*}
Otherwise, assume $m=1$ and $n=-1$, we have
\begin{align*}
[L_1, L_{-1}]
&=[\widehat{\cL_1}, \widehat{\cL_{-1}}]\\
&=\widehat{\{\cL_1, \cL_{-1}\}}+C(h_1, h_{-1})\\
&=2\widehat{\cL_0}-{1\over2}\sum_{a}(\mu_a^+-{1\over 2})(\mu_a^++{1\over 2})\\
&=2L_0.
\end{align*}
Notice that here we use \eqref{cocycle}.
\qed

%\newpage
\section{Semi-simple Frobenius manifolds and Virasoro constraints}
% LG mirror symmetry}
\label{sec-max}

In this section, we first review Givental Theorem on the Virasoro constraints for semi-simple Frobenius manifolds, Theorem \ref{semi-simple-Virasoro}.
Then we search some semisimple Frobenius manifolds and apply the Givental Theorem to prove Theorem \ref{thm-max} and Proposition
\ref{two-fermat-semisimple}. %and Theorem \ref{general-pillow}.
%We also discuss the mirror symmetry of LG models and its connection to Virasoro constraints.
%We will focus on CohFTs in quantum singularity theories with underlying Frobenius manifolds generically semi-simple.
%One can try to directly . For example, there will be no

\iffalse
Virasoro operators can be introduced on both LG models.
In particular, we consider the Virasoro constraints for a special type of admissible LG pair $(W, G)$, with $W$ an {\em invertible polynomials}, and $G$ the group of diagonal symmetries $G_W$.
Mirror symmetry between LG models \cite{BH, Kr} predicts that such an admissible pair $(W, G_W)$ is mirror to another pair $(W^T, \{1\})$, where $W^T$ is also an invertible polynomial, and $\{1\}\leq G_{W^T}$ is the group of identity element.
The LG pair $(W^T, \{1\})$ is not admissible.
However, there is a deformation theory for $W^T$ as a {\em nondegenerate singularity} using Saito's primitive forms, which captured the genus zero date of a CohFT, and Givental-Teleman formula applies to those generically semi-simple Frobenius manifolds. The CohFT in result is called a Saito-Givental B-model for the singularity $W^T$.

In this section, we will prove that the Virasoro constraints for the admissible invertible pair $(W, G_W)$ is a consequence of mirror symmetry and Givental's result on the Virasoro constraint for semi-simple CohFTs.

\begin{itemize}
\item
Mirror models and Mirror Virasoro constraints.
\item
semi-simple cases.
\end{itemize}
\fi

%\subsection{Givental-Teleman formula and Saito's Frobenius manifold}

\subsection{%semi-simple Frobenius manifolds and
Frobenius manifolds and Givental Theorem}
\label{sec-givental-thm}
%In this section, we review Givental's Theorem on Virasoro constraints for semi-simple Frobenius manifolds \cite{G-ham}.
\subsubsection{Conformal Frobenius manifold, quantum connection, and calibration}
Dubrovin \cite{Du} introduced the concepts of {\em Frobenius manifold} to study the geometry of 2D topological field theories.
%\begin{definition} \cite{Du}
Briefly speaking, a Frobenius manifold $M$ consists of a quadruple
$\Big(M, \<\cdot,\cdot\>, F, \one \Big)$
with
\begin{itemize}
\item
a flat metric $\<\cdot ,\cdot \>$ on the tangent bundle $T_tM$ for $t\in M$;
\item
a {\em prepotential} $F(t)$, whose $3$-rd covariant derivative gives an {\em associative multiplication} $\star_t$
$$F_{abc}=\<a\star_t b, c\>=\<a, b\star_t c\>;$$
\item a flat vector field $\one$, which is an identity of the multiplication $\star_t$.
\end{itemize}
%\end{definition}
%Furthermore, $M$ is called {\em conformal} if there exists .

A Frobenius manifold is called {\em conformal} if there exists an {\em Euler vector field}.
See \cite{Du} for the definitions and details of these concepts.
Examples of conformal Frobenius manifolds exist in Gromov-Witten theory, Fan-Jarvis-Ruan-Witten theory, and Saito's construction of primitive forms for miniversal deformations of isolated critical points of holomorphic functions \cite{Sai}. % In this section, we will focus on the later two.

\begin{remark}
Dubrovin's notion of Frobenius manifolds can be generalized to Frobenius super-manifolds, where the quantum multiplication $\star_t$ becomes super-commutative.
The examples appear in this section are all Frobenius manifolds, while some examples in Section \ref{sec-calabi-yau} are Frobenius super-manifolds, see Example \ref{bigrforell}.
\end{remark}

Let $\{\phi_a\}$ be a homogeneous flat basis of the Frobenius manifold and $t_0^a$ be the coordinate of $\phi_a$. There exists a {\em quantum connection}
\begin{equation}
\label{quantum-connection}
\nabla_z:=d-z^{-1}\sum_{a}(\phi_a\star_t )dt_0^a\wedge.
\end{equation}
Following \cite{Du, G-ham}, one can construct a fundamental solution $S_{t}(z)$, called {\em calibration}, which is upper-triangular as in \eqref{upper-triangular}.
Such a calibration is unique up to right multiplication by a constant lower upper-triangular operator $C(z)$ which does not depend on the choice of $t$.

%Using the {\em topological recursion relation} \cite[Theorem 4.2.9]{FJR}, it is standard to check that the

%\subsubsection{} Saito's Frobenius manifolds. $\big({\rm Jac}(f), {\rm Res}_f, \big)$

\subsubsection{Givental Theorem}
%\subsubsection{A higher genus formula for semi-simple CohFT}

A Frobenius manifold is called {\em semi-simple} at $t\in M$ if the Frobenius algebra $\star_t$ is semi-simple.
Let $M$ be a Frobenius manifold of (complex) dimension $r$. If $M$ is semi-simple at $t\in M$, then there exist a {\em canonical coordinate system}: $\mathbf{u}=(u_1, \ldots, u_r)$ near $t\in M$, such that
$${\pop u_i}\star_t {\pop u_j}=\delta_{i,j}{\pop u_i}.$$
Let $U:={\rm diag}(u_1, \ldots, u_r)$ and $\Psi_t$ be the base change from $\{{\pop u_i}\}$ to $\{\sqrt{\Delta_i}{\pop u_i}\}$, where
$$\<{\pop u_i}, {\pop u_j}\>={\delta_{i, j}\over\Delta_i}.$$

%The following result is proved by Dubrovin \cite{Du}.
According to \cite{Du, G-higher, T},  we consider a Frobenius manifold of conformal dimension $K$ with an Euler vector field
 $$E=\sum_{a}\left(1-d_a\right)t_0^a{\partial\over \partial t_0^a}+\sum_{a; d_a=1}\rho_a{\partial\over \partial t_0^a},$$
 there exists a unique lower triangular symplectic transformation $R_t(z)$ such that
\begin{equation}\label{R}
[R_t(z), z^{-1}E\star]=(z\partial_z+\widetilde{\theta})R_t(z),
\end{equation}
and
\begin{equation}\label{R and S}
%S_\tau(z)\sim
S_t(z)\sim\Psi_t\, R_t(z)\, e^{U/z}.
\end{equation}
Here $\widetilde{\theta}:={\rm diag}(d_1-K/2, \ldots, d_r-K/2)$.
%is a solution of the quantum differential equation \eqref{quantum-connection}.

%Define $\mathcal{A}_\tau$ by translation.
Let $\mathcal{D}_{pt}(\hbar; \mathbf{q})$ be the Witten-Kontsevich $\tau$-function.
For semisimple Frobenius manifolds, Givental \cite{G-higher} constructed the {\em (abstract) total descendent potential}
\begin{equation}
\label{abstract-descendent}
\mathcal{D}(\hbar; \mathbf{q}):=e^{F^1(t)}\widehat{S}_t^{-1}\widehat{\Psi_t}\widehat{R_t}e^{\widehat{U/z}}\prod_{i=1}^N\mathcal{D}_{pt}(\hbar\Delta_i; \sqrt{\Delta_i}\mathbf{q}^i).
\end{equation}
and the {\em (abstract) total ancestor potential}
\begin{equation}
	\label{gt-formula}
	\mathcal{A}_t(\hbar; \mathbf{q})=%e^{-F^1(t)}e^{C(t)}
	\widehat{\Psi_t}\widehat{R_t}e^{\widehat{U/z}}\prod_{i=1}^N%\mathcal{D}_{pt}(\hbar; \mathbf{q}^i).
	\mathcal{D}_{pt}(\hbar\Delta_i; \sqrt{\Delta_i}\mathbf{q}^i).
	\end{equation}
Here  $q^a_m=t^a_m-\delta_{m, 1}\delta_{a, 0}$ is the dilaton shift, and $F^1(t)$ is the genus one potential. %, and $C(t)$ is determined by the matrix $R_1$ in $R_t(z)$.

%\subsubsection{Givental's Theorem}
%\textcolor{red}{Nothing is original in this section.}
%In this section, we recall Givental's theorem on Virasoro constraints for semi-simple CohFTs \cite{G-ham}.
%The result is well known to experts.
%The main tool is Givental formalism.
%{\color{red}I think this part is well-known to Dubrovin and Givental's school. We should not claim any credit on this part.}

\iffalse
For a conformal semi-simple Frobenius manifold $(\cH, \star, (,), E)$, if the elements $\phi_a\in \cH$ has bigrading $(\mu_a^+, \mu_a^-)$, the Euler operator $E$ can be expressed as
$$E=(1-\deg\phi_a)\tau^a\partial_{\tau_a}=(1-\mu_a^++{\widehat{c}_W\over 2})\tau^a\partial_{\tau_a}=-\widetilde{E}|_{\tau_{\geq0}^a=0},$$
where $\widetilde{E}$ is the grading operator defined in \eqref{bigE}.
Consider
	$$\nabla_\tau:=\partial_z+z^{-1}\theta+z^{-2}E\star,$$
	we have
	$$\nabla_0=\partial_z+z^{-1}\theta.$$
Now the Virasoro operators are of the form
	$$L_k=(z^{-1/2}(z\nabla_0z)^{k+1}z^{-1/2})^{\widehat{}}-\delta_{k,0}\cdot{1\over 4}{\rm Str}(\theta^2-{1\over 4}).$$
Introduce $S$ operator from quantum differential equations. Introduce descedent potential.
{\red Add description for absolute descedent potential, which works for semi-simple Frobenius manifold in general.}
\fi

According to Kontsevich Theorem \cite{Kon} on the Witten conjecture \cite{Wit-conj}, the Witten-Kontsevich $\tau$-function $\mathcal{D}_{pt}(\hbar; \mathbf{q})$ is annihilated by differential operators
$$%\widetilde{L}_k:=
(z^{-1/2}(z\partial_z z)^{k+1}z^{-1/2})^{\widehat{}}+{\delta_{0,k}\over 16}, \quad k\geq -1.$$
For a conformal semisimple Frobenius manifold, let $\rho=\sum_{a; d_a=1}\rho_a{\partial\over \partial t_0^a}$ in the Euler vector field $E$.
%$$E=\sum_{a}\left(1-d_a\right)t_0^a{\partial\over \partial t_0^a}+\sum_{a; d_a=1}\rho_a{\partial\over \partial t_0^a}.$$
Replacing the auxiliary operator in \eqref{auxiliary-operator} by $z(\partial_z+z^{-1}\widetilde{\theta})z+\rho$, and defining differential operators by \eqref{virasoro-equality}, which we denote by $L_k^{\rm Giv}$ here, Givental proved an equality between two operators in \cite[Theorem 8.1]{G-ham} and deduced the following result from Kontsevich Theorem.
\begin{theorem}
\label{semi-simple-Virasoro}
 \cite[Proposition 7.7]{G-ham}
The Virasoro operators $\{L_k^{\rm Giv}\}_{k\geq -1}$ annihilate the total descendant potential $\mathcal{D}(\hbar; \mathbf{q})$ of a semi-simple Frobenius manifold.
% $(\cH, \star, (,), E)$.
That is,
	$$L_k^{\rm Giv}\mathcal{D}(\hbar; \mathbf{q})=0, \quad k\geq -1.$$
\end{theorem}
%The proof is well known to experts.
%\textcolor{red}{For the readers convenience, we present a proof in Appendix \ref{sec-appendix}. Some of the arguments will be used in Section ?.}
We call it {\em Givental Theorem on the Virasoro constraints for semi-simple Frobenius manifolds}, or simply {\em Givental Theorem} in this paper.

\subsubsection{Givental-Teleman Theorem}
A CohFT is called semisimple if the underlying Frobenius manifold is semisimple.
The following result is a consequence of the famous Givental-Teleman Theorem \cite{T}, which says the semisimple CohFT is uniquely reconstructed from the underlying Frobenius manifold.
\begin{theorem}\label{GT}
\cite{T}
For a semi-simple CohFT, the abstract total ancestor potential is the same as the formal total ancestor potential defined from the CohFT.
\end{theorem}
See \eqref{formal-fjrw} and \eqref{GW-ancestor} for examples of formal total ancestor potentials.

%\subsection{Invertible polynomail and maximal groups}

\subsubsection{Frobenius manifolds in LG A-models}
In LG A-models, there is a conformal formal Frobenius (super)-manifolds (see \cite[Corollary 4.2.8]{FJR} for example)
$$\Big(\cH^{\clubsuit}_{W, G}, (\cdot, \cdot), F_0^{\clubsuit, (W, G)}, 1|J\>, E\Big).$$
Here
\begin{itemize}
\item the metric is given by the pairing $(\cdot,\cdot)$ in \eqref{qst-pairing};
\item  the homogeneous basis $\{\phi_a=[f_a]d\bx_{\gamma_a}\}$ of $\cH_{W,G}$ is a flat basis;
\item the prepotential is given by the genus zero potential
$$F_0^{\clubsuit, (W, G)}(t)=\sum_{n\geq3}{1\over n!}\LD t, \cdots, t\RD_{0,n}^{\clubsuit, (W, G)}, \quad t=\sum\limits_{a}t_0^a\phi_a.$$
%F^0\left(\sum\limits_{a}t_0^a\phi_a\right).$$
\item
the Euler vector field $E$ is of the form
\begin{equation}
\label{euler-givental}
E=\sum_{a}\left(1-\deg_\C\phi_a\right)t_0^a{\partial\over \partial t_0^a}.
\end{equation}
%where $\deg_\C\phi_a$ is the
%Frobenius
%degree of the element $\phi_a$, defined in \eqref{complex-degree}.
\item the conformal dimension $K$ is the central charge $\widehat{c}_W$.
\end{itemize}
The quantum product $\star_t$ is giving by genus-zero three-point correlators
\begin{equation}
\label{quantum-product}
(\phi_a\star_t\phi_b, \phi_c)={\partial^3\over \partial t_{0}^{a}\partial t_{0}^{b}\partial t_{0}^{c} }F_0^{\clubsuit, (W, G)}(t)
:=\LL\phi_a, \phi_b, \phi_c\RR_{0,3}^{\clubsuit}(t).
\end{equation}

For arbitrary $t$, one can construct a formal total ancestor potential from the CohFT
\begin{equation}
\label{formal-fjrw}
\mathcal{A}^{\clubsuit}_{W,G}(t)
:=\exp\left(\sum_{g}\hbar^{2g-2}\sum_{k}{1\over k!}\LL\prod_{i=1}^{k}\bt(\psi_i)
%\bt(\psi_1), \cdots, \bt(\psi_k)
\RR_{g,k}^{\clubsuit, (W, G)}(t)\right).
\end{equation}
The quantum invariants in the definition of  the formula \eqref{ancestor-qst} are replaced by the double brackets, which are formal power series of $t$,
$$\LL\prod_{i=1}^{k}\bt(\psi_i)\RR_{g,k}^{\clubsuit, (W, G)}(t):= \sum_{m\geq 0} {1\over m!}\LD\prod_{i=1}^{k}\bt(\psi_i), t, \cdots, t\RD_{g,k+m}^{\clubsuit, (W, G)}.$$
\begin{remark}
We make a few remarks when applying the results in Section \ref{sec-givental-thm} to the LG A-models in Section \ref{semisimple-mirror} and Section \ref{sec-nonmax}.
\begin{enumerate}
\item  According to Givental-Teleman Theorem \ref{GT}, if the Frobenius manifold is generically semisimple at $t$,  then the abstract total ancestor potential  $\mathcal{A}_t(\hbar; \mathbf{q})$ in \eqref{gt-formula} is the same as the geometric total ancestor potential $\mathcal{A}^{\clubsuit}_{W,G}(t)$ in \eqref{formal-fjrw}.
 \item According to \cite{Mil}, the potential $\mathcal{A}^{\clubsuit}_{W,G}$ in \eqref{ancestor-qst} is the limit of the geometric total ancestor potential $\mathcal{A}^{\clubsuit}_{W,G}(t)$ in \eqref{formal-fjrw} at $t=0$.
\item We can choose the calibration operator $S_t(z)$ by
$$(S_{t}(z)(\phi_a), \phi_b):=(\phi_a, \phi_b)+\LL\frac{\phi_a}{z-\psi_1}, \phi_b\RR_{0, 2}(t).$$
Taking the limit $t\to 0$, we see the total descendent potential $\mathcal{D}(\hbar; \mathbf{q})$ in \eqref{abstract-descendent} in LG A-model theories is the same as the potential $\mathcal{A}^{\clubsuit}_{W,G}$, up to a constant.
\item In Section \ref{semisimple-mirror} and Section \ref{sec-nonmax}, the LG A-model theories we consider satisfies the following condition,
\begin{equation}\label{bigrade-eq}
\mu_a^+=\mu_a^-, \quad \forall \phi_a\in \mathcal{H}_{W, G}.
\end{equation}
See Lemma \ref{bigrade-Gmax}. %and Lemma \ref{bigrad-2v}}.
\item If \eqref{bigrade-eq} holds, then $\mu_a^+=\deg_\C\phi_a-\widehat{c}_W/2$. Since the term $\rho$ in the Euler vector field vanishes, the Virasoro operator $L_k^{\rm Giv}$ in Theorem \ref{semi-simple-Virasoro} matches the Virasoro operator $L_k$ defined in \eqref{virasoro-operator}.
\end{enumerate}
\end{remark}
%\begin{corollary}
Thus if the Frobenius manifold $(\cH^{\clubsuit}_{W, G}, (\cdot, \cdot), F_0^{\clubsuit, (W, G)}, 1|J\>, E)$ is generically semisimple, Conjecture \ref{VC} is a consequence of Theorem \ref{semi-simple-Virasoro}.
%\end{corollary}
Next we will give some examples.

\subsection{Semisimple Frobenius manifolds via LG mirror symmetry}
\label{semisimple-mirror}
%\subsubsection{LG mirror symmetry}
%We start with the definition of an invertible polynomial.
According to the mirror symmetry statements proved in \cite{HPSV} and \cite{HLSW}, we obtain a large class of semisimple Frobenius manifolds in LG A-model theories from admissible LG pairs with an invertible polynomial and the maximal group.
\begin{definition}
[Invertible polynomials]
A polynomial $W$ is called {\em invertible} if up to a rescaling of the variables, $W$ can be written in the form
\begin{equation}
\label{invertible-poly}
W=\sum_{i=1}^{n}\prod_{j=1}^{n}x_j^{a_{ij}}
\end{equation}
and its \emph{exponent matrix} $E_W:=\left(a_{ij}\right)_{n\times n}$ is an invertible matrix in ${\rm GL}(n, \mathbb{Q})$.
\end{definition}
According to the classification by Kreuzer and Skarke \cite{KS},
an invertible polynomial must be a direct sum of atomic types ($a\geq2, a_i\geq2$):
\begin{enumerate}
\item \emph{Fermat atomic type:}	  $x^a.$
\item \emph{Chain atomic type:}	$x_1^{a_1}x_2 + x_2^{a_2}x_3 + \ldots + x_{k-1}^{a_{k-1}}x_k + x_k^{a_k}.$
\item \emph{Loop atomic type:}	         $x_1^{a_1}x_2 + x_2^{a_2}x_3 + \ldots + x_k^{a_k}x_1.$
\end{enumerate}
We slightly abuse the notation and still call $W$ is of {\em Fermat type} if
\begin{equation}
\label{fermat-type}
W=x_1^{a_1}+\cdots +x_n^{a_n}.
\end{equation}
\begin{lemma}\label{bigrade-Gmax}
Let $(W, G_W)$ be an admissible LG pair and $W$ be an invertible polynomial. For any $\phi_a\in \mathcal{H}_{W, G_W}$, we have $\mu_a^+=\mu_a^-$.
\end{lemma}
\begin{proof}
If $W$ is a direct sum of $W_1, \ldots, W_N$, then $G_W=G_{W_1}\times \cdots \times G_{W_N}$. Now it suffices to prove this lemma in the case $W$ is of atomic type. By \eqref{bigrading}, for $\phi_a=\alpha|\gamma\>$, we only need to consider the case $\phi_a$ is broad, and prove that $N_\gamma=2{\rm wt}(\alpha)$. For $\phi_a$ is broad, by \cite[Lemma 1.7]{Kr}, we have
\begin{enumerate}
\item If $W$ is Fermat type, then there is no broad element.
\item If $W$ is loop type with $k$ variables, then  $\gamma={\rm Id}$, $k$ is even, and
$$\alpha=\prod_{i=1}^k x_i^{\delta_i^{\rm even}(a_i-1)}dx_i \quad \text{or} \quad \prod_{i=1}^k x_i^{\delta_i^{\rm odd}(a_i-1)}dx_i.$$
\item If $W$ is chain type with $k$ variables, then $N_\gamma$ is even, and
$$\alpha=\prod_{k-N_\gamma+1}^k x_i^{\delta_{k-i}^{\rm odd}(a_i-1)}dx_i.$$
\end{enumerate}
Keeping in mind $W$ is of weight $1$, one can check ${\rm wt}(\alpha)=N_{\gamma}/2$ easily.
\end{proof}

%$$W=\sum_{i=1}^{N}\prod_{j=1}^{N}x_j^{a_{ij}}.$$
%Berglund-H\"ubsch \cite{BH} constructed a mirror polynomial for an invertible polynomial.
 If $W$ is an invertible polynomial of the form \eqref{invertible-poly}, Berglund-H\"ubsch \cite{BH} constructed a mirror polynomial
$$
W^T=\sum_{i=1}^{n}\prod_{j=1}^{n}x_j^{a_{ji}}.
$$
So the exponent matrix $E_{W^T}$ of $W^T$ is the transpose matrix of $E_W$.

Now we recall two mirror symmetry statements proved in \cite{HLSW} and  \cite{HPSV}.
\begin{theorem}
[All genus mirror symmetry for invertible polynomials]
\label{mirror them}
\cite{HLSW, HPSV}
\label{thm-HPSV}
Let $(W, G_W)$ be an admissible LG pair and $W$ be an invertible polynomial. There exists a primitive form $\zeta$ and a bigrading preserving mirror map $\cH^A_{W, G_W}\to {\rm Jac}(W^T)$. Under the mirror map, there are :
\begin{enumerate}
\item (Polishchuk-Vaintrob to Saito-Givental mirror theorem \cite[Theorem 1.3]{HPSV})
$$\mathscr{A}^{\rm PV}_{W,G_W}=\mathscr{A}^{\rm SG}_{W^T, \zeta}.$$
\item (Fan-Jarvis-Ruan-Witten to Saito-Givental mirror theorem \cite[Theorem 1.2]{HLSW})
If $W$ has no weight-${1\over 2}$ chain variable,
$$\mathscr{A}^{\rm FJRW}_{W,G_W}=\mathscr{A}^{\rm SG}_{W^T, \zeta}.$$
\end{enumerate}
\end{theorem}
Here  $\mathscr{A}^{\rm SG}_{W^T, \zeta}$ is the Saito-Givental limit potential for $W^T$, see \cite{HLSW} for the details. %\eqref{SG-limit}.
The main ingredients in the proof in \cite{HLSW} and \cite{HPSV} are the identification between the Frobenius manifold of the admissible LG A-model $(W, G_W)$ and the Saito's Frobenius manifold of the isolated singularity $W^T$ \cite{Sai}, with an appropriate choice of a primitive form $\zeta$ \cite{LLS, LLSS}..
% the authors use sophisticated reconstruction via WDVV equations and explicit calculations of certain genus-zero three-point and four-point correlators on both A/B sides to prove ,
% if one choose a primitive form $\zeta$ appropriately \cite{LLS, LLSS}.
%By considering these Frobenius manifolds as limits of the deformed semi-simple Frobenius manifolds \cite{Mil}, and applying Givental-Teleman formula \ref{GT}.

The condition in Part (2) is not essential.
It is caused by the lack of enough computation tools for the analytic part of FJRW theory.
The condition can be removed in some cases of exceptional singularities, when $W=x^2+xy^2+yz^4$ and $W=x^2+xy^3+yz^3$ \cite[Section 2.3]{LLSS}.\\

\subsubsection{A proof of Theorem \ref{thm-max}}
It is well known that Saito's Frobenius manifold of $W^T$ is generically semisimple.
Now Theorem \ref{thm-max}  follows from Theorem \ref{thm-HPSV} and Givental Theorem \ref{semi-simple-Virasoro}.
\qed

%Mirror symmetry can be formulated for invertible admissible LG pairs $(W, G)$ even when then group $G$ is not the maximal group. %$G\neq G_W$.
%Berglund-Henningson \cite{BHe} constructed a mirror pair $(W^T, G^T)$ and Krawitz\cite{Kr} proved the state spaces are isomorphic $\cH_{W, G}\cong\cH_{W^T, G^T}$.
%If $G=G_W$, the mirror group $G^T$ is trivial and the situations are discussed in Section \ref{semisimple-mirror}.
%If $G\neq G_W$, then $G^T$ is not trivial, and a Frobenius manifold has not been constructed for such LG B-model pair $(W^T, G^T)$ in general.

\subsection{Semisimplicity via quantum Euler vector field}

%\textcolor{red}{State it in semi-simple cases and then discuss why the A-models are semi-simple, via mirror symmetry.}
%We are ready to prove our main theorem in this section, which is due to Givental\cite{Gi}:
\label{sec-nonmax}

In this section, we consider some A-model LG pairs $(W,\<J\>)$ studied in \cite{Francis}, where the group $\<J\>$ is not the maximal group $G_W$.
We consider the following list
\begin{equation}
\label{2-variable-cases}
x^4+y^4, x^3+y^6, x^3+y^9, x^4+y^6, x^3+xy^8.
\end{equation}
The main result of this section is
\begin{proposition}
\label{two-fermat-semisimple}
Let $W$ be an invertible polynomial in \eqref{2-variable-cases}.
Then the FJRW theory of the admissible LG pair  $(W, \<J_W\>)$ has the following properties:
\begin{enumerate}
\item
the Frobenius manifold is generically semisimple near the origin of $\cH_{W, \<J\>}$;
\item
the Virasoro Conjecture \ref{VC} holds true.
\end{enumerate}
\end{proposition}
We see part (2) is Proposition \ref{general-pillow}.
It follows from part (1) and Givental Theorem \ref{semi-simple-Virasoro}.
In the remaining part of this section, we prove part (1) using the criteria described in \cite{Abr}.
\begin{remark}
Our method works in much more generality for invertible polynomials with two-variables.
In general, the formulas involved will be more complicated, even for the other two examples $x^3+xy^6$ and $x^3y+y^7$ in \cite{Francis}.
%The result will appear in a joint work with Amanda Francis.
\end{remark}

\subsubsection{Quantum product $\star_t$}
Let $\mu=\dim_\C\cH_{W, \<J\>}$.
We fix a basis of $\cH_{W, \<J\>}$, denoted by $\{\phi_i\mid i=1, \ldots, \mu\}$.
We write the dual basis by $\{\phi^i\mid i=1, \ldots, \mu\}$ where $\<\phi_i, \phi^j\>=\delta_i^j.$
%Now we assume $W$ is a two-variable invertible polynomial in \eqref{2-variable-cases}.
If $J^k\in \<J\>$ is a narrow element, the generator $1\vert J^k\>\in \cH_{J^k}$ is also denoted by the symbol $e_{kJ}$ in \cite{Francis}, where $e_0=dxdy$.
We fix a basis of $\cH_{W, \<J\>}$ as follows. % that contains $e_{2J}$ and parametrize it by $t$ as follows.
The notations here is from \cite{Francis}, except we don't use $W$ to avoid confusion.
We always have $\phi_1=\one=e_{J}, \phi_2=e_{(d-1)J}, \phi_3=e_{2J}.$

%

%We will compute the vector field by restricting to ${\bf v}=t \cdot e_{2J}$.

\begin{table}[H]
\label{basis-two-variable}
\caption{A table of basis of $\cH_{W, \<J\>}$}
\

\centering

\begin{tabular}{|c|c|| l |}
  \hline
$W$ &$\mu$ & $\phi_1, \ldots, \phi_\mu$\\
\hline
$x^4+y^4$ & $6$ & $\one, e_{3J}, e_{2J}, 4y^2e_0=4X, 4xye_0=4Y, 4x^2e_0=4Z$\\
  \hline
$x^3+y^6$ & $6$ & $\one, e_{5J}, e_{2J}, Z, 3X, 6Y$\\
\hline
$x^3+y^9$ & $8$ & $\one, X^2Y, Y, X, XY, X^2, 3Z, 9xy^2e_0$\\ \hline
$x^4+y^6$ & $9$ & $\one, XY^2, Y, Y^2, X^2, Z, X, XY, \sqrt{24}xy^2e_0$\\ \hline
$x^3+xy^8$ & $10$ & $\one, X^3Y, Y, X, XY, X^2, X^2Y, X^3, \sqrt{24}xy^3e_0, \sqrt{-8} Z$\\ \hline
%$$ & $$ & $$\\ \hline
\end{tabular}
\end{table}

Let us denote the dimension of the narrow subspace of $\cH_{W, \<J\>}$ by $\mu_{\rm nar}$, the weights of the variables by $({w_1\over d}, {w_2\over d})$, where $d, w_1, w_2$ are positive integers and $\gcd(w_1, w_2)=1$.

%For example, for the LG pair $(x^4+y^4, \<J\>)$, we can fix a basis
%$$\{\phi_i\}=\{e_J=\one, e_{3J}=W^2, e_{2J}=W, 4y^2e_0=4X, 4xye_0=4Y, 4x^2e_0=4Z\}.$$
%Each basis is dual to itself.
%The broad elements are arranged after narrow elements.
We parametrize $e_{2J}$ by $t$ and compute the quantum product $\star_{t}$.
We give an example to compute $e_{(d-1)J}\star_t$. It is enough to find all nontrivial invariants of the form
\begin{equation}
\label{7-point}
\<e_{(d-1)J}, \phi_i,\phi_j, e_{2J}, \ldots, e_{2J}\>_{0,n+3}.
\end{equation}
Since
$$\deg_\C e_{(d-1)J}=2-2{w_1+w_2\over d}=\widehat{c}_W, \quad \deg_\C e_{2J}={w_1+w_2\over d}.$$
The degree constraint \eqref{deg-fjrw} implies if the invariant in \eqref{7-point} does not vanish, then
$$\deg_\C\phi_i+\deg_\C\phi_j={n\over 2}\widehat{c}_W.$$
Since $0\leq \deg_\C\phi_i, \deg_\C\phi_j \leq \widehat{c}_W$, we must have $n=0,1,2,3,4.$
In fact, if $n=1$ or $3$, there is no nontrivial invariant of the form \eqref{7-point}.
Now we list all possible nontrivial contributions as follows.
Some of the FJRW invariants (including the $7$-point invariant) are obtained by WDVV equations.
\begin{enumerate}
\item if $n=0$, then $\<e_{(d-1)J}, \one, \one\>_{0,3}=1$;
\item if $n=2$, then
$$\<e_{(d-1)J}, \phi_i, \phi^i, e_{2J}, e_{2J}\>_{0,5}=
\begin{dcases}
0, & i=1, 2;\\
{w_1w_2\over d^2}, & 3\leq i \leq \mu_{\rm nar};\\
-{w_1w_2\over d^2}, & i>\mu_{\rm nar}.
\end{dcases};$$
\item if $n=4$, then by WDVV equation, we have
\begin{eqnarray*}
&&\<e_{(d-1)J}, e_{(d-1)J}, e_{(d-1)J}, e_{2J}, e_{2J}, e_{2J}, e_{2J}\>_{0,7}\\
&=&{4\choose 2} \left(\<e_{(d-1)J}, \phi_i, \phi^i, e_{2J}, e_{2J}\>_{0,5}\right)^2\\
&=&{6w_1^2w_2^2\over d^4}.
\end{eqnarray*}
\end{enumerate}

Let ${\rm Id}_i$ be the identity operator on $\C\cdot\phi_i$.
Then the multiplication of $e_{(d-1)J}\star_t$ is given by the matrix
\begin{equation}
\label{multiply-top}
\begin{pmatrix}
0 & {w_1^2w_2^2\over 4d^4}t^4\\
1 & 0
\end{pmatrix}
\bigoplus \left(\bigoplus_{i=3}^{\mu_{\rm nar}} {w_1w_2 \over 2d^2}t^2{\rm Id}_i \right)\bigoplus\left(\bigoplus_{i=\mu_{\rm nar}+1}^{\mu} -{w_1w_2\over 2d^2}t^2{\rm Id}_i\right).
\end{equation}

\subsubsection{A quantum Euler vector field}
Let ${\bf v}=\sum_it_i\phi_i$, we consider the {\em  quantum Euler vector field}
%We consider a {\em quantum Euler vector} ${\bf E}({\bf v})$ defined by
\begin{equation}
{\bf E}({\bf v}):=\sum_{i=1}^{\mu}\phi_i\star_{\bf v}\phi^i.
%\label{euler-two-fermat}{\bf E}({\bf v}):=\sum_{k=1}^{d-1}\alpha_k\star_{\bf v}\alpha^k+\sum_{j=1}^{d-1}\beta_j\star_{\bf v}\beta^j.
\end{equation}
Let us calculate the restriction ${\bf E}(t):={\bf E}(t\cdot e_{2J})$. We give an example when the polynomial is $x^4+y^4.$
 Using the FJRW invariants listed in \cite[Page 1365]{Francis}, we obtain
\begin{eqnarray*}
{\bf E}(t)&=&\sum_{i=1}^{6}\LL\phi_i,\phi^i,\one\RR_{0,3}(t)\cdot e_{(d-1)J}
+\left(\LL  e_{2J}, e_{2J}, e_{3J}\RR_{0,3}(t)+\LL 4X, 4Z, e_{3J}\RR_{0,3}(t)\right.\\
&&\left.+\LL 4Y, 4Y, e_{3J}\RR_{0,3}(t)+\LL 4Z, 4X, e_{3J}\RR_{0,3}(t)\right)\cdot\one\\
&=&6 e_{3J}-{t^2\over 16}\one.
\end{eqnarray*}
The calculation for other cases are similar.
\begin{lemma}
Let $W$ be an invertible polynomial in \eqref{2-variable-cases}.
The quantum Euler vector filed of the FJRW theory of $(W, \<J\>)$ along $e_{2J}$ has the form
\begin{equation}
\label{euler-formula}
{\bf E}(t)=\mu \cdot e_{(d-1)J}+(2\mu_{\rm nar}-2-\mu){w_1w_2 t^2\over 2 d^2}\one.
\end{equation}
\end{lemma}
The proof of this Lemma is tedious but direct, only using the FJRW invariants listed in \cite{Francis} and WDVV equations.
Let us remark that if $W$ is of Fermat type, neither 4-point invariants nor 6-point invariants has nontrivial contributions.
On the other hand, for $x^3+xy^8$, both $\<Y,Z,Z,Z\>$ and $\<Y,Z,xy^3e_0,xy^3e_0\>$ has nontrivial contribution.
However, they cancel with each other when we compute ${\bf E}(t)$ because $$\LL\sqrt{-8}Z, \sqrt{-8}Z, \sqrt{-8}Z\RR_{0,3}(t)+\LL\sqrt{24}xy^3e_0, \sqrt{24}xy^3e_0, \sqrt{-8}Z\RR_{0,3}(t)=0.$$

\subsubsection{A proof of Proposition \ref{two-fermat-semisimple}, Part (1):}

According to \cite[Theorem 3.4]{Abr}, the Frobenius manifold is semi-simple at ${\bf v}$ if and only if
$$\det {\bf E}({\bf v})\star_{\bf v}\neq 0.$$
In fact, we just need to restrict the multiplication to ${\bf v}=t\cdot e_{2J}$.
%\begin{remark}
%Calculate the bigrading. All the elements have even parity.
%\end{remark}
We see the quantum multiplication of second term in \eqref{euler-formula} on the fixed basis gives a multiple of identity matrix.
Using \eqref{multiply-top} and \eqref{euler-formula}, the multiplication ${\bf E}(t)\star_t$ on the basis
$\{\phi_1=\one, \phi_2=e_{(d-1)J}, \phi_3, \cdots, \phi_{\mu}\}$
is given by the matrix
$$M \bigoplus\left( \bigoplus_{i=3}^{\mu_{\rm nar}} {(\mu_{\rm nar}-1)w_1w_2\over d^2}t^2{\rm Id}_i\right) \bigoplus \left(\bigoplus_{i=\mu_{\rm nar}+1}^{\mu} -{(\mu_{\rm nar}-1-\mu)w_1w_2\over d^2}t^2{\rm Id}_i\right),$$
where
$$M=\begin{pmatrix}
{(2\mu_{\rm nar}-2-\mu)w_1w_2\over 2 d^2}t^2 & {w_1^2w_2^2\mu\over 4d^4}t^4\\
\mu & {(2\mu_{\rm nar}-2-\mu)w_1w_2\over 2 d^2}t^2
\end{pmatrix}.$$
Thus we have
$
\det {\bf E}(t)\star_t%=\det M
\neq0
$
if $t\neq0$.
Thus $\det {\bf E}({\bf v})\star_{\bf v}\neq 0$ in a small neighborhood of ${\bf v}=(0, 0, t, 0, \cdots, 0)$ and the result follows
from  \cite[Theorem 3.4]{Abr}.
\qed

\section{Virasoro constraints and LG/CY correspondence}
%Landau-Ginzburg/Calabi-Yau correspondence}
\label{sec-calabi-yau}
In this section, we consider the Virasoro constraints for admissible LG pairs $(W, G)$ when $W$ is of Calabi-Yau type.
\begin{definition}
[Calabi-Yau polynomials]
We say a quasihomogeneous polynomial $W(x_1, \cdots, x_n)$ is of {\em Calabi-Yau type} if
%the central charge is a positive integer. That is,
\begin{equation}
\label{cy-condition}
\widehat{c}_W:=\sum_{i=1}^{n}(1-2q_i)
=n-2\in \mathbb{Z}_+.
\end{equation}
\end{definition}
%For an admissible LG pair $(W, G)$ with $W$ of Calabi-Yau type.
Let $TX_W$ be the tangent bundle of the hypersurface
$$X_W:=(W=0)\subset \mathbb{P}^{n-1}(w_1, \cdots, w_n).$$
%in the weighted projective space $$.
The condition \eqref{cy-condition} implies that $c_1(TX_W)=0$.
On the other hand, the finite group $\widetilde{G}:=G/\<J_W\>$ acts on the hypersurface $X_W$ and the global quotient
$$X_{W,G}:= X_W/\widetilde{G}$$ also satisfies $c_1(TX_{W,G})=0$.
We will abuse the notations to call these varieties $X_{W,G}$ of Calabi-Yau type.

Let $(W, G)$ be an admissible LG pair of CY type. The LG A-model theories of $(W, G)$ are closely related to the Gromov-Witten theory for the Calabi-Yau variety $X_{W,G}$.
Such connections are called {\em Landau-Ginzburg/Calabi-Yau correspondence} in the literature \cite{CR}, inspired by the work of physicists (see \cite{Wit} for the references therein).

In general, LG pairs of CY type are not generically semi-simple so Givental's Theorem is no longer applicable. However, the idea of LG/CY correspondence is very useful here.

If the Calabi-Yau $X_{W,G}$ has complex dimension at least three, (or $\widehat{c}_W\geq 3$ equivalently), then Virasoro Conjecture \ref{VC} can be obtained by a simple degree calculation.
We show it in Section \ref{sec-central-large}.
When the Calabi-Yau side $X_{W,G}$ is an elliptic curve, we can prove Virasoro Conjecture \ref{VC} using the Virasoro constraints for the elliptic curve in \cite{OP} and the Landau-Ginzburg/Calabi-Yau correspondence studied in \cite{LLSZ, LSZ}.
%An example will be discussed in Section \ref{fermat-cubic}.
%The pillowcase LG pair will be revisited in
In Section \ref{sec-simple-elliptic}, we discuss the cases for all the pairs $(W, G)$ when $X_{W,G}$ is one-dimensional Calabi-Yau.

%\newpage
\subsection{Calabi-Yau polynomials with large central charges}
\label{sec-central-large}
For admissible LG pairs of CY type, if the central charge is at least three, as an analog of the Virasoro constraints for Calabi-Yau manifolds that has been proved in \cite[Theorem 7.1]{Ge}, we have
\begin{lemma}\label{thm-lc}
Let $(W, G)$ be an admissible LG pair of Calabi-Yau type with $\widehat{c}_W\geq 3$.
If for any nonzero $\phi_a=\alpha|\gamma\>\in \cH_{W, G}$ with
$\gamma\neq J$, we have $\deg_\C\phi_a\geq 1$,
%\begin{equation} \label{integer-degree} \deg_\C\phi_a={\mu_a^++\mu_a^-+\widehat{c}_W\over 2}\geq 1,\end{equation}
then the Virasoro constraints
$$L_k\cA_{W, G}^{\clubsuit}=0, k\geq -1$$
is equivalent to Assumption \ref{alge cycle}.
	\end{lemma}
%\begin{proof}Here $\frac{1}{2}(\mu_a^++\mu_a^-+\widehat{c}_W)$ is the A-model degree of $\phi_a$ (see \cite{Kr}). The proof is along the line of Theorem 7.1 in \cite{Ge}.\end{proof}
Now we deduce the second part of Theorem \ref{thm-cy} from this Lemma.
\begin{corollary}\label{virasoro-high CY}
	Let $W$ be an invertible polynomial of Calabi-Yau type. If $\widehat{c}_{W}\geq 3$ and $G\leq {\rm SL}_n(\C)$, then Virasoro Conjecture \ref{VC} holds for $\mathcal{A}_{W,G}^{\rm FJRW}$.\end{corollary}
	\begin{proof}
For a fixed homogeneous element $\phi_a=\alpha|\gamma\>\in \cH_{W, G}$ with
$$\gamma=\left(\exp(2\pi\sqrt{-1}\theta_1), \cdots, \exp(2\pi\sqrt{-1}\theta_n)\right)$$
the group element given in \eqref{group-exponent}.
Recall that ${\rm age}(\gamma)=\sum\limits_{i=1}^{n}\theta_i$
is defined \eqref{degree-shift}.
%Recall ${\rm age}(\gamma)$ defined \eqref{degree-shift},
Then $G\leq {\rm SL}_n(\C)$ implies that ${\rm age}(\gamma)\in \mathbb{Z}_{\geq0}$.
Using the bigrading formula \eqref{bigrading} and  the CY condition \eqref{cy-condition}, we have
\begin{align*}
\deg_\C\phi_a %=\frac{1}{2}(\mu_a^++\mu_a^-+\widehat{c}_W)
%={N_\gamma\over 2}+\iota_\gamma
={N_\gamma\over 2}+{\rm age}(\gamma)-\sum_{i=1}^nq_i={N_\gamma\over 2}+{\rm age}(\gamma)-1.
\end{align*}
%Since $W$ is Calabi-Yau, we have $\sum_{i=1}^nq_i=1$.

If $\gamma=J^0$, then ${\rm age}(\gamma)=0$ and $N_\gamma=n=2+\widehat{c}_W\geq 5$. So $\deg_\C\phi_a \geq 1$ holds.

If $\gamma\neq J^0$, then ${\rm age}(\gamma)\geq 1$. The condition $\deg_\C\phi_a \geq 1$ could fail if % in two situations:
		\begin{enumerate}
			\item${\rm age}(\gamma)=1$ and $N_\gamma=0$;
			\item ${\rm age}(\gamma)=1$ and $N_\gamma=1$.
		\end{enumerate}	
		
For case (1), we claim that $\gamma=J$. In fact,
%suppose $$\gamma=(\exp(2\pi  \sqrt{-1} q_1'), \cdots, \exp(2\pi \sqrt{-1} q_n'),$$
$N_\gamma=0$ implies  for each $i$, we have $0<\theta_i<1$. %Set $\mathbf{q}'=( q_1', \ldots, q_n')$,
Recall $E_W$ is the exponent matrix of $W$, then $\gamma\in G_W$ implies
$$E_W
\begin{pmatrix}
\theta_1\\
\vdots\\
\theta_n
\end{pmatrix}\in \Z_+^n.$$
%is a $\N$-valued matrix}.
Let $q_i^T$ be the weight of $x_i$ in the mirror  polynomial $W^T$, then we have
$$
\begin{pmatrix}
q_1^T&\cdots&q_n^T
\end{pmatrix}
E_W
\begin{pmatrix}
\theta_1\\
\vdots\\
\theta_n
\end{pmatrix}
=
\sum_{i=1}^{n}\theta_i=1.$$
%$$\mathbf{\widehat{q}}\cdot E_W \cdot \mathbf{q}'^T=\sum_{i=1}^nq_i'=1.$$
Since $q_i^T$ are positive numbers and
$$\sum_{i=1}^{n}q_i^T=\sum_{i=1}^{n}q_i=1,$$
we must have
$$E_W
\begin{pmatrix}
\theta_1\\
\vdots\\
\theta_n
\end{pmatrix}
=\begin{pmatrix}
1\\
\vdots\\
1
\end{pmatrix}.
$$
%$$E_W\cdot (\theta_1, \cdots, \theta_n)^T=(1, 1, \cdots, 1)^T.$$
This implies $\theta_i=q_i$ for each $i$ and $\gamma=J$.

On the other hand, case (2) is impossible. In fact, $W_\gamma$ is still an invertible polynomial so we must have $W_\gamma=x_i^{a_i}$ for some $i$.
Then it is easy check that there is no $\<J\>$-invariant element in ${\rm Jac}(W_\gamma)dx_\gamma$. %so there is no $\phi$ satisfying case (2).

In conclusion, if $\gamma\neq J$, condition $\deg_\C\phi_a\geq 1$ always holds and the result follows from Lemma \ref{thm-lc}.
	\end{proof}

\iffalse
	
	Furthermore
If $\gamma\neq J^0$,  then ${\rm age}(\gamma)\geq 1$.
Now we have
$$\deg_\C\phi_a={N_\gamma \over 2}+{\rm age}(\gamma)-1 \geq 0.$$
	Notice that the condition $\widehat{c}_{W}\geq 3$ implies $n\geq 5$.	If $d_\phi<1$, $n\geq 5$ implies that $\gamma\neq {\rm Id}$, then we have two possibilities:
	
	Now one can apply Theorem \ref{thm-lc} to complete the proof.
	\fi

%\newpage
\subsection{Fermat CY polynomials of three variables}
\label{fermat-cubic}
Fermat CY polynomials with three variables can be written in the form of
\begin{equation}
\label{fermat-1-fold}
W_d:=x_1^{a_1}+x_2^{a_2}+x_3^{a_3}, \quad \textit{with}\quad {1\over a_1}+{1\over a_2}+{1\over a_3}=1.
\end{equation}
Here we assume $d=a_1\geq a_2\geq a_3$, then we have
$$(a_1, a_2, a_3)=(3,3,3), (4,4,2), \textit{or }  (6,3,2).$$
%We write $d:=a_1$ in each case.
The FJRW theory of the admissible LG pairs
$$\Big(W_d:=x_1^{a_1}+x_2^{a_2}+x_3^{a_3}, \<J\>\Big)$$
%$W$ is a Fermat CY polynomial of three variables and $G$ is generated by $J_W$.
has been studied in \cite{LSZ}.
When $d=3,4,6$, we choose
$$h(W_d):=x_1x_2x_3/27,\quad x_1^2x_2^2/32,\quad x_1^4x_2/36.$$
We form a homogeneous basis of $\cH_{W_d, \<J\>}$ by
$$\Big\{1|J\>, 1|J^2\>, d{\bf x}|J^0\>, h(W_d)d{\bx}|J^0\>\Big\}.$$
We parametrize them by $t_0^0, t_0^1, s_0^1,$ and $s_0^0$ respectively.
The choice of $h(W_d)$ is made to match the Poincar\'e pairing on the GW side.

%Recall $(\ell)_k$ is the Pochhammer symbol defined in \eqref{pochhammer}.
According to \eqref{virasoro-operator}, the Virasoro operators for the pair $(W_d, \<J\>)$ are %(c.f. Example \ref{bigrforell} for the cubic case)
	\begin{align}
	L_k:=&-(k+1)!{\pop t_{k}^{0}}\nonumber\\
	&+\sum_{m\geq 0}^\infty\left((m)_{k+1}t^0_m\frac{\partial}{\partial t^0_{k+m}}+(m+1)_{k+1}t^1_m\frac{\partial}{\partial t^1_{k+m}}\right)\label{cubic-virasoro}\\
	&+\sum_{m\geq 0}^\infty\left((m+1)_{k+1}s^0_m\frac{\partial}{\partial s^0_{k+m}}+(m)_{k+1}s^1_m\frac{\partial}{\partial s^1_{k+m}}\right).\nonumber
	\end{align}

\iffalse	
Recall the deformed total ancestor potential $\cA_u$ is defined by
\begin{align*}
\cA_{u}&:=\exp\left(\sum_{g\geq0}\hbar^{g-1}\cF_{g}({\bf u})\right)\\
&:=\exp\left(\sum_{g\geq0}\hbar^{g-1}\sum_{\substack{n; \\ 2g-2+n>0}}\LL {\bf u}, \cdots, {\bf u} \RR^{\circ}_{g,n}(u)\right).
\end{align*}
\fi

Recall $\mathcal{A}^{\rm FJRW}_{W,G}$ is the ancestor FJRW potential for the LG pair $(W,G)$ defined in \eqref{ancestor-qst}.
In this section, we will prove Part (2) of Theorem \ref{thm-cy}, which is
\begin{theorem}
\label{virasoro-cubic}
Let $W_d$ be a Fermat CY polynomial in \eqref{fermat-1-fold} and $L_k$ be the Virasoro operators in \eqref{cubic-virasoro}.
%Let $\cA^{\rm FJRW}_{W_d, \<J\>}(\tau)$ be the ancestor FJRW potential of the pair $(W_d, \<J\>)$, then
For all $k\geq -1$, we have Virasoro constraints
$$L_k \mathcal{A}^{\rm FJRW}_{W_d,\<J\>}=0.$$
\end{theorem}

We will prove Theorem \ref{virasoro-cubic} in three steps:
\begin{enumerate}
\item We recall the Virasoro constraints for descendent Gromov-Witten potential in \cite{OP}.
\item We use the Ancestor/Descedant correspondence (Theorem \ref{thm-ancestor-descendant}) to write the Virasoro constraints for the ancestor Gromov-Witten potential of the elliptic curves.
\item We use the LG/CY correspondence proved in \cite{LSZ} to complete a proof for Theorem \ref{virasoro-cubic}.
\end{enumerate}
%For simplicity, we restrict the Virasoro operators to even part in the current version.
%{\color{red}I think the restriction is not necessary, since Ginvental formalism work for super state space. In the following argument I will include the odd part.}

\subsubsection{Ancestor and descendent}
For a compact K\"ahler manifold $X$, there are two types of Gromov-Witten invariants, called {\em ancestor invariants} and {\em descendent invariants}, depends on the choice of psi-classes.

Let $\{\alpha_i\in H^*(X, \mathbb{C})\}$ be a set of cohomology classes of $X$.
Let $\overline{\cM}_{g,n}(X, d)$ be the moduli stack of degree-$d$ stable maps from a connected genus $g$ curve with
$n$ markings to the target $X$.
The moduli stack has a virtual fundamental cycle, denoted by $[\overline{\cM}_{g,n}(X, d)]^{\rm vir}$.
%Let $\overline{\cM}_{g,k,d}(\sE)$ be the moduli space of degree $d$ stable maps to the elliptic curve $\sE$.
Let $\pi: \overline{\cM}_{g,k}(X,d)\to \overline{\cM}_{g,k}$ be the forgetful morphism and ${\rm ev}_{i}: \overline{\cM}_{g,k}(X,d)\to X$ be the evaluation morphism given by the $i$-th marking.

The descendent GW invariants are defined by intersecting the virtual fundament cycles in GW theory with psi classes $\{\bar\psi_i\}$ on moduli space of stable maps.
Let $\{\alpha_a\}$ be a homogeneous basis of $H^*(X, \mathbb{C}).$
Define the {\em total descendent GW potential} of $X$ by
$$\cD^X=\exp\left(\sum_{g\geq 0}\hbar^{2g-2}\sum_{k, d}{Q^d\over k!}\int_{\left[\overline{\cM}_{g,k}(X, d)\right]^{\rm vir}}\prod_{i=1}^{k}
\sum_{m=0}^{\infty}\left(\widetilde{t}_{a}^{m}{\rm ev}_{i}^{*}(\alpha_a)\wpsi_i^m\right)\right).
$$
%Let ${\rm ev}_k,k=1,2,\cdots, n$ be the evaluation morphisms, $\pi$ be the forgetful morphism, and

On the other hand, similar to the ancestor invariants \eqref{ancestor-inv} in LG A-model theories, the ancestor GW invariants are defined by intersecting the Gromov-Witten CohFT classes with psi classes on $\overline{\cM}_{g,k}$.
We denote by the \emph{ancestor GW invariants}
 % intersection numbers over $\overline{\M}_{g,n}(\E_d, \beta)$
\begin{equation*}
\<\alpha_1\psi_1^{\ell_1},\cdots,\alpha_n\psi_n^{\ell_n}\>^{X}_{g,n,d}
=\int_{[\overline{\cM}_{g,n}(X,d)]^{\rm vir}}\prod_{k=1}^n{\rm ev}_k^*(\alpha_k)\pi^*\psi_k^{\ell_k}\,.
\end{equation*}
Let $$t:=\sum_{a}t^a\alpha_a\in H^*(X, \mathbb{C}), \quad \widetilde{\bt}(z):=\sum_{m\geq 0}\sum_{a}\wt_m^a\alpha_a\,z^m.$$
The \emph{ancestor GW correlation function}
\begin{equation}
 \label{GW-function}
 \LL\alpha_1\psi_1^{\ell_1},\cdots,\alpha_k\psi_k^{\ell_k}\RR_{g,k}^{X}(t)
 =\sum_{d}\sum_{m\geq0}{Q^d\over m!} \LD\alpha_1\psi_1^{\ell_1},\cdots,\alpha_k\psi_k^{\ell_k}, t, \cdots, t\RD_{g,k+m,d}^{X}\,.
 \end{equation}
 Via the dilaton shift
 $$\widetilde{\bq}(z)=\widetilde{\bt}(z)-\one\cdot z,$$
we denote the {\em total GW ancestor potential} of the target $X$ by
\begin{equation}
\label{GW-ancestor}
\mathcal{A}^{X}_{t}(\widetilde{\bq}):=\exp\left(\sum_{g}\hbar^{2g-2}\sum_{k}{1\over k!}\LL\widetilde{\bt}(\psi), \cdots, \widetilde{\bt}(\psi)\RR^{X}_{g, k}(t)\right).
\end{equation}
The total ancestor potential depends on a choice of $t$ while $\cD^X$ does not.

%We define an operator
%$$S^{X}_{t}(z)=1+S_1z^{-1}+S_2z^{-2}+\cdots\in {\rm End}(H)[\![z^{-1}]\!]$$
Using {\em topological recursion relations}, the operators $S^{X}_{t}$ defined by
\begin{equation}
\label{calibration}
(S^{X}_{t}(z)(\alpha_1), \alpha_2):=(\alpha_1, \alpha_2)+\<\!\<\frac{\alpha_1}{z-\bar\psi_1}, \alpha_2\>\!\>^{X}_{0, 2}(t)
\end{equation}
is a calibration with respect to the quantum connection in \eqref{quantum-connection}.
The operator $S^{X}_t(z)$ is a symplectic transformation.
%The series $S_{t}(z)$ depends on the parameter $t\in \cH$.
%The two total potentials $\cD^X(\widetilde{\bq})$ and $\cA^X_t(\widetilde{\bq})$ are related by the following formula.
%Coates-Givental \cite{CG} prove the following theorem for any CohFT, based on the topological recursion relation \eqref{TRR}:
\begin{theorem}
\cite{KM98, G-ham}
\label{thm-ancestor-descendant}
Let $F_1^{X}(t):=\LL\RR^X_{1,0}(t)$ be the genus-one GW generating function, then
	$$\mathcal{D}^X(\widetilde{\bq})=e^{F^{X}_1(t)}\widehat{S^X_t}^{-1}\mathcal{A}^X_t(\widetilde{\mathbf{q}}).$$
\end{theorem}

\subsubsection{Virasoro constraints for the descendents of elliptic curves}
Let $W_d$ be the Fermat CY polynomial in \eqref{fermat-1-fold}.
The hypersurface determined by the $W_d=x_1^{a_1}+x_2^{a_2}+x_3^{a_3}$ in the weighted projective space $\mathbb{P}^2\left({d\over a_1}, {d\over a_2}, {d\over a_3}\right)$ is an elliptic curve
$$\sE_d:=X_{W_d}=(W_d=0)\subset\mathbb{P}^2\left({d\over a_1}, {d\over a_2}, {d\over a_3}\right).$$
The Gromov-Witten theory does not depend on the choice of the elliptic curves. So we drop the subscript in $\sE_d$ and consider a basis of cohomology $H^*(\sE, \C)$, given by
%\textcolor{red}{Explain the bigrading here, which is not from Hodge decomposition?}
\begin{itemize}
\item the identity class $\one\in H^0(\sE, \C)$,
\item the Poincar\'e dual of the point $\omega\in H^2(\sE, \C)$,
\item the classes $\alpha, \beta\in H^{1}(\sE, \C)$, which is  a symplectic basis of $H^1(\sE, \C)$.
\end{itemize}
We assign a bigrading and a parity for the basis $\{1, \alpha ,\beta, \omega\}$ as below.
The bigrading is the Hodge grading shifted by $-1/2$. %and the parity is the cohomological parity $(-1)^{p+q}$ for elements in $H^{p, q}(\sE, \C)$.
\begin{table}[H]
\caption{State space of the elliptic curve}
\begin{center}
\begin{tabular}{|c|c|c|c|c|}
  \hline
 % element
  $\phi$ &  $1$ &    $\omega$  & $\alpha$ &   $\beta$\\
  \hline
 % bigrading
 $(\mu_\phi^+,\mu_\phi^-)$ & $(-{1\over 2}, -{1\over 2})$ & $({1\over 2}, {1\over 2})$   & $({1\over 2}, -{1\over 2})$ & $(-{1\over 2}, {1\over 2})$ \\
  \hline
 % parity
 $|\phi|$ & $1$ & $1$ & $-1$ & $-1$\\
  \hline
\end{tabular}
\end{center}
\label{table:elliptic}
\end{table}

Let
$\Psi: H^{*}(\sE, \C)\to \cH_{W_d, \<J\>}$
be a linear map defined by
\begin{equation}
\label{cubic-isom}
\Psi(\one)=1|J\>, \quad \Psi(\omega)=1|J^{-1}\>, \quad \Psi(\alpha)=h(W_d)d\bx, \quad \Psi(\beta)=d\bx.
\end{equation}
Comparing Table \ref{table-cubic} with Table \ref{table:elliptic}, we see  the isomorphism $\Psi$ preserves bigrading and parity.

We use the ordered set of variables
$\Big\{\widetilde{t}^0_{k}, \widetilde{s}^0_{k}, \widetilde{s}^1_{k}, \widetilde{t}^1_{k}\Big\}$
to parametrize both the descendent insertions
$$\{\one\wpsi^k, \alpha\wpsi^k, \beta\wpsi^k, \omega\wpsi^k\},$$
and the ancestor insertions
$$\{\one\psi^k, \alpha\psi^k, \beta\psi^k, \omega\psi^k\}.$$
Using our terminology, the Virasoro operators $\{L^{\sE}_k\}_{k\geq -1}$ for the target elliptic curve $\sE$ are given by
\begin{align}
L_k^{\sE}=&-(k+1)!{\partial\over\partial \widetilde{t}_{k+1}^0}\nonumber\\
&+\sum_{\ell\geq0}\left((\ell)_{k+1}\widetilde{t}_{\ell}^0 {\partial\over \widetilde{t}_{k+\ell}^0}+(\ell+1)_{k+1}\widetilde{t}_\ell^1{\partial\over \partial \widetilde{t}_{k+\ell}^{1}}\right)\label{viraroso-op}
\\&+\sum_{\ell\geq0}\left((\ell+1)_{k+1}\widetilde{s}_{\ell}^{0} {\partial\over \widetilde{s}_{k+\ell}^{0}}+(\ell)_{k+1}\widetilde{s}_\ell^1{\partial\over \partial \widetilde{s}_{k+\ell}^{1}}\right).\nonumber
\end{align}
Similar as Proposition \ref{Virasoro-relation}, the operators satisfy
$$
[\widetilde{L}_n, \widetilde{L}_m]=(n-m) \widetilde{L}_{n+m}.
$$
%Here the supertrace formula ${\rm Str}(\theta^2-{1\over 4})=0$ is easily checked.

In fact, under the linear isomorphism $\Psi: H^{*}(\sE, \C)\to \cH_{W_d, \<J\>}$ defined by \eqref{cubic-isom}, we can identify the two Virasoro operators
$$L_k(\bt)=L_k^{\sE}(\widetilde{\bt})\vert_{\widetilde{\bt}=\bt}.$$
%Let  $\mathcal{L}_k$ be the operators defined in \eqref{virasoro-diff} and $\widehat{\cL_k}$ be its quantization operators.
%By Proposition \ref{Virasoro-quantization}, {\red if we identify the coordinates}, we have
%\begin{lemma}The two systems of differential operators coincide:$$\widetilde{L}_k=L_k, \quad \textit{for all } k\geq -1.$$\end{lemma}

The Virasoro constraints for the (descendent) Gromov-Witten theory of the elliptic curve are solved by Okounkov and Pandharipande \cite{OP}.
The following result is a special case of \cite[Theorem 3]{OP}, which is deduced from \cite[Theorem 3]{OP} by summing over all degree $d$ in the absolute theory.
\begin{proposition}
\cite{OP}
\label{descedant-virasoro}
The Virasoro constraints  for the absolute descendent Gromov-Witten theory of the elliptic curve $\sE$ hold:
\begin{equation}
\label{virasoro-curve}
L^{\sE}_k \cD^\sE=0, \quad \textit{for all }  k\geq -1.
\end{equation}
\end{proposition}

\subsubsection{Virasoro constraints for ancestor GW theory of elliptic curves}
\iffalse
For ancestor invariants, we define Virasoro operators
\begin{align}
L_k^{E}=&-(k+1)!{\partial\over\partial t_{k+1}^0}\nonumber\\
&+\sum_{\ell\geq0}\left((\ell)_{k+1}t_{\ell}^0 {\partial\over \partial t_{k+\ell}^0}+(\ell+1)_{k+1}t_\ell^1{\partial\over \partial t_{k+\ell}^{1}}\right)\label{viraroso-op}\\
&+\sum_{\ell\geq0}\left((\ell+1)_{k+1}s_{\ell}^{0} {\partial\over \partial s_{k+\ell}^{0}}+(\ell)_{k+1}s_\ell^1{\partial\over \partial s_{k+\ell}^{1}}\right)\nonumber
\end{align}
{\color{red}I think it is more convenient to use one $L^E_k$ to denote virasoro operators of ancestor and descendant potential, since they are just the same, and more compatiable with \eqref{ancestor-descendant}.}
\fi

\iffalse
suppose $\tau_1, \tau_\alpha, \tau_\beta, \tau_\omega$ is the parameter corresponding to $1, \alpha ,\beta, \omega$ on the primary state space $H$ of GW theory of $E$, we consider a special slice
$$H_0:=\{\tau_1= \tau_\alpha=\tau_\beta=0, \tau_\omega=\tau\}\subset H.$$
By restricting Theorem \ref{thm-ancestor-descendant} to $H_0$, the Ancestor/Descendant correspondence of the elliptic curve is given by
\begin{equation}
\label{ancestor-descendant}
\mathcal{D}^{E}=e^{\mathcal{F}^{\circ E}_{1}(\tau)}\widehat{S}_{\tau}^{-1} \mathcal{A}_{\tau}^{E}\,,
\end{equation}
under the identification
$\tilde{t}^i_\ell=t^i_\ell, \tilde{s}^i_\ell=s^i_\ell$\,.
\fi

According to Theorem \ref{thm-ancestor-descendant}, the descendent potential $\mathcal{D}^{\sE}$ does not depend on the choice of $t$.
We restrict to the slice $t=t_0\omega$, where the genus one potential $F^{\sE}_1(t_0)$ and the operator $S^{\sE}(t_0)$ can be computed explicitly for the elliptic curve $\sE$.
We make a coordinate change
$$q:=e^{t_0}.$$
From now on, we write
$$\begin{dcases}
F^{\sE}_1(q):=F^{\sE}_1(t_0)|_{t_0=\log q}, \\
S^{\sE}(q):=S^{\sE}(t_0)|_{t_0=\log q}.
\end{dcases}$$

Now we use Theorem \ref{thm-ancestor-descendant} and Proposition \ref{descedant-virasoro} to prove
\begin{proposition}
\label{ancestor-cubic-virasoro}
The Virasoro constraints
\begin{equation}
\label{virasoro-ancestor}
L_k^{\sE} \mathcal{A}^{\sE}_{q}=0, \quad \forall k\geq -1,
\end{equation}
holds for the ancestor GW theory of the elliptic curves.
\end{proposition}
\begin{proof}
%Here $\mathcal{F}^{\circ E}_{1}(\tau)$ is the genus one generating function of connected GW invariants of the elliptic curve $E$ and
We have
$$\mathcal{F}^{\sE}_{1}(q)=-\log (q)_{\infty}= -\log \prod_{n=1}^{\infty}(1-q^n).$$
Thus $\exp(\mathcal{F}^{\sE}_{1}(q))$ commutes with the Virasoro operators
\begin{equation}
\label{f1-commute}
L_k^{\sE} =e^{-\mathcal{F}^{\sE}_{1}(q)} L_k^{\sE} \, e^{\mathcal{F}^{\sE}_{1}(q)}.
\end{equation}
%We denote the infinitesimal symplectic operator $\log S_{\tau}$ by $$\log S_{\tau}:={h(\tau)\over z}.$$

Next we consider the operator $S^{\sE}(q)$. Using the formula of $S^X_{t}$ in \eqref{calibration} and fixing a basis $(\one, \omega, \alpha, \beta)\in H^*(\sE, \C)$, we obtain that $\log S^{\sE}(q)\in {\rm End}(H^*(\sE, \C))$ is given by the matrix
\begin{equation}\label{S-cubic}
	 %h(\tau)(\one)=\tau \omega, h(\tau)(\omega)=h(\tau)(\alpha)=h(\tau)(\beta)=0.
	 %h(\tau)
	 %\Big(\log S^{\sE}(q)\Big)
	 %\begin{pmatrix}\one & \omega &\alpha &\beta\end{pmatrix}
	 %=\begin{pmatrix}\one & \omega &\alpha &\beta\end{pmatrix}
	 \begin{pmatrix}
	 0&0&0&0\\
	 {q\over z}&0&0&0\\
	 0&0&0&0\\
	 0&0&0&0
	 \end{pmatrix}.
\end{equation}
%Via \eqref{S-cubic}, we have $$S_{\tau}= 1+{h(\tau)\over z}.$$
Thus the quadratic Hamiltonian for $\log S_{\tau}$ is given by
\begin{equation}
\label{quadratic-calibration}
\Omega(\log S(q) f, f)=-q{(\tilde{q}_0^0)^2\over 2}-q\cdot\sum_{k\geq 0} \tilde{q}^0_{k+1} \tilde{p}^{0}_k\,.
\end{equation}
%Similar as the proof of Theorem \ref{G_W case}, we have to calculate
Now we calculate $$S(q)\circ \mathcal{L}_k \circ S(q)^{-1}$$ and the cocycle that appears in the conjugation.
%Firstly, similar as the proof of Lemma \ref{conjugation}, set

By \eqref{S-cubic}, we have
$$S(q)\circ (\partial_z+z^{-1}\theta) \circ S(q)^{-1}=\partial_z+z^{-1}\theta.$$
Then \eqref{virasoro-diff} implies that,
$$S(q)\circ \mathcal{L}_k \circ S(q)^{-1}=\mathcal{L}_k.$$
Since the term \eqref{quadratic-calibration} only contains $(\tilde{q}_0^0)^2$-term but $\widetilde{\mathcal{L}}_k$ does not have any $p_ap_b$-terms, all cocycles vanish by Lemma \ref{key-cocycle}. We obtain that the quantization operator $\widehat{S}(q)$ commutes with the Virasoro operators
$$\widehat{S}(q) L^{\sE}_k \,\widehat{S}(q)^{-1}= L^{\sE}_k.$$
Combining \eqref{f1-commute}, we obtain
%Proposition \ref{descedant-virasoro} makes us to consider:
\begin{align*}
L_k^{\sE} \cA_\tau^{\sE}
&=\widehat{S}(q) L^{\sE}_k \,\widehat{S}(q)^{-1}\cA_\tau^{\sE}\\
&=\widehat{S}(q) e^{-\mathcal{F}^{\sE}_{1}(q)} L_k^{\sE} \, e^{\mathcal{F}^{\sE}_{1}(q)} \,\widehat{S}(q)^{-1}\cA_\tau^{\sE}\\
&=\widehat{S}(q) e^{-\mathcal{F}^{\sE}_{1}(q)} L_k^{\sE} \, \cD^{\sE}\\
&=0.
\end{align*}
The last two equations follow from Theorem \ref{thm-ancestor-descendant}  and Proposition \ref{descedant-virasoro}.
\end{proof}

\subsubsection{Virasoro constraints for Fermat LG pairs $(W_d, \<J\>)$}
Now we consider the last step in the proof of Thereom \ref{virasoro-cubic} using LG/CY correspondence in \cite{LSZ}.
We remark that the LG/CY correspondence does depend on the choice of the defining polynomial.
The approach in \cite{LSZ} relates the GW theory and LG theory by  holomorphic Cayley transformations of quasi-modular forms.
Let us recall the construction and quasi-modular forms briefly \cite{KZ, Zag, SZ}.

Let $\tau$ be the coordinate of the upper half plane.
Let $\Gamma:={\rm PSL}(2, \mathbb{Z})$ be the modular group.
Let $E_{2k}(\tau)$ be the weight-$2k$ Eisenstein series, and
$$\widetilde{M}(\Gamma):=\C[E_2(\tau), E_4(\tau), E_6(\tau)]$$
be the ring of quasi-modular forms \cite{KZ}.
There is a natural ring isomorphism, called {\em modular completion},  from $\widetilde{M}(\Gamma)$ to
$$\widehat{M}(\Gamma):=\C[\widehat{E}_2(\tau, \bar\tau), E_4(\tau), E_6(\tau)],$$
the ring of almost holomorphic modular forms,
%$$\widetilde{M}(\Gamma)\cong \widehat{M}(\Gamma):=\C[\widehat{E}_2(\tau, \bar\tau), E_4(\tau), E_6(\tau)],$$
by sending the ring generators $E_4(\tau)$ and $E_6(\tau)$ to themselves, and $E_2(\tau)$ to its modular completion
$$\widehat{E}_2(\tau,\bar\tau):=E_2(\tau)-{3\over \pi {\rm Im}(\tau)}=E_2(\tau)-{6\sqrt{-1}\over \pi (\tau-\bar\tau)}.$$

Fixing a complex multiplication point $\tau_*$ in the upper half plane, and a constant $c\in\C$, there is a coordinate change, called {\em Cayley transformation},
\begin{equation}
\label{cayley}
s=c\cdot {\tau -\tau_*\over \tau-\bar{\tau_*}}.
\end{equation}
This coordinate change induces a linear operator on $\widehat{M}(\Gamma)$, denoted by $\mathscr{C}_{\tau_*}$, which sends a weight $2k$ almost holomorphic modular form $\widehat{f}(\tau, \bar\tau)\in\widehat{M}_{2k}(\Gamma)$,
to analytic function of $s$ and its complex conjugate $\bar{s}$, denoted by
$$\mathscr{C}_{\tau_*}(\widehat{f})(s, \bar{s})=(2\pi\sqrt{-1}c)^{-k}\left({\tau(s)-\bar\tau_*\over\tau_*-\bar\tau_*}\right)^{2k}\widehat{f}(\tau(s), \bar\tau(s)),$$
Here $\tau(s)$ is the inverse of \eqref{cayley}, and the operator $\mathscr{C}_{\tau_*}$ is called the {\em Cayley transformation on
$\widehat{M}(\Gamma)$}.
The modular completion $\widetilde{M}(\Gamma)\to \widehat{M}(\Gamma)$ has an inverse, called a {\em holomorphic limit}, by taking $\lim_{\bar\tau\to\infty}$, or equivalently $q\to 0$.
A similar notion of holomorphic limit can be defined for $\mathscr{C}_{\tau_*}(\widehat{f})(s, \bar{s})$, and one obtain a {\em holomorphic Cayley transformation} on $\widetilde{M}(\Gamma)$ \cite{Zag, SZ}, denoted by $\mathscr{C}^{\rm hol}_{\tau_{*}}$, via the commutative diagram

\begin{figure}[H]
  \renewcommand{\arraystretch}{1}
\begin{displaymath}
\xymatrixcolsep{4pc}\xymatrixrowsep{4pc}\xymatrix{  \widetilde{M}(\Gamma)
\ar@/^/[r]^{\textrm{modular completion}}\ar@{.>}[d]^{\mathscr{C}_{ \tau_*}^{\rm hol} }
 & \widehat{M}(\Gamma) \ar[d]^{\mathscr{C}_{\tau_*}} %\curvearrowleft \widehat{\partial}_{\tau}
 \\
 \mathscr{C}^{\rm hol}_{\tau_{*}}(\widetilde{M}(\Gamma)) &  \ar@/^/[l]^{\textrm{holomorphic limit}}  \mathscr{C}_{\tau_*}(\widehat{M}(\Gamma))
 %\curvearrowleft  \widehat{\partial}_{s}
 }
\end{displaymath}
 %\caption[InvarianceofRamanujan]{Cayley transformation on quasi-modular and almost-holomorphic modular forms.}
  \label{figureinvarianceRamanujan}
\end{figure}

For each $d=3, 4, 6$, we make a coordinate change
$$q=\exp\left({2\pi\sqrt{-1}\tau\over d}\right).$$
Recall that the ancestor GW functions are Fourier series of quasi-modular forms \cite{OP, LSZ}
$$\LL\alpha_1\psi_1^{\ell_1}\cdots \alpha_n\psi_n^{\ell_n}\RR_{g, n}^{\sE_d}(q)\in \widetilde{M}(\Gamma).$$
The following LG/CY correspondence is proved in \cite{LSZ} for the Fermat cubic LG pair %$(W_3, \<J\>)$.
$$\Big(W_3=x_1^3+x_2^3+x_3^3, \<J\>\Big).$$
\begin{theorem}
\cite[Theorem 1]{LSZ}
\label{thm-lg-cy}
Let
$$\Psi: H^*(\sE_3, \C)\to \cH_{W_3,\<J\>}$$ be the linear isomorphism defined in \eqref{cubic-isom}.
Let $\tau_*=-{\sqrt{-1}\over 3}\exp\left({2\pi\sqrt{-1}\over 3}\right).$
There exists a holomorphic Cayley transformations $\mathscr{C}^{\rm hol}_{\tau_*}$, such that
$$\mathscr{C}^{\rm hol}_{\tau_*}\left(\LL\alpha_1\psi_1^{\ell_1}\cdots \alpha_n\psi_n^{\ell_n}\RR_{g, n}^{\sE_3}(q)\right)=\LL \Psi(\alpha_1)\psi_1^{\ell_1}\cdots \Psi(\alpha_n)\psi_n^{\ell_n}\RR_{g, n}^{\rm FJRW}(s).$$
%Here  is a point in the upper half plane, and the change of flat coordinates are given by a Cayley transformation \eqref{cayley} for more general $\tau_*$.
\end{theorem}
Here $s$ is the parameter of $1|J^2\>\in \cH_{W_3, \<J\>}$.
The holomorphic Cayley transformation sends the ancestor GW functions to Taylor series of $s$, near $s=0$, which is $\tau=\tau_*$.
The resulting Taylor series are exactly the corresponding ancestor FJRW functions.
We can rewrite Theorem \ref{thm-lg-cy} in terms of ancestor potentials by
\begin{equation}
\label{cayley-ancestor}
\mathscr{C}^{\rm hol}_{\tau*}(\cA^{\sE_3}_{q}(\widetilde{\bq}))=\cA_s^{W_3,\<J\>}(\bq).
\end{equation}

{\em A proof of Theorem \ref{virasoro-cubic}:}
We deal with the Fermat cubic pair $(W_3, \<J\>)$ first.
By the construction of the holomorphic Cayley transformation, we see it commutes with the Virasoro operators
$$L_k\mathscr{C}^{\rm hol}_{\tau*}=\mathscr{C}^{\rm hol}_{\tau*}L_k^{\sE_3}.$$
Now applying \eqref{cayley-ancestor} and \eqref{virasoro-ancestor}, we have
\begin{align*}
L_k\left(\cA_s^{W_3,\<J\>}(\bq)\right)
=L_k\mathscr{C}^{\rm hol}_{\tau*}(\cA^{\sE_3}_{q}(\widetilde{\bq}))
=\mathscr{C}^{\rm hol}_{\tau*}L_k^{\sE_3}(\cA^{\sE_3}_{q}(\widetilde{\bq}))=0.
\end{align*}

\iffalse
The same argument as Proposition \ref{ancestor-cubic-virasoro} show that
$$L_k\cA_t^{W,G}=0, \quad k\geq -1.$$
Now observe that Virasoro operators $\{L_k\}, \{L_k^E\}$ is independent of the parameter $t$ and $\tau$ respectively, and they equal with each other when we identify the variables via $\Psi$. So holomorhic Cayley transformation commutes with the Virasoro operators:
$$L_k\mathscr{C}_{\tau*}=\mathscr{C}_{\tau*}L_k^E.$$
Thus Theorem \ref{virasoro-cubic} follows from Proposition \ref{ancestor-cubic-virasoro}. {\color{red}More explicitly, consider the following two set of function with respect to variable $\tau$ and $t$ respectivly:
\begin{align*}
	L_k^E\mathcal{A}^E_\tau, \quad L_k\cA_t^{W,G}, \quad k\geq -1.
\end{align*}
By \cite[Theorem 1]{LSZ}, we have the second one is the analytic continuation of the first one via $\Psi$ and Cayley transformation, so the first one vanishes implies the second vanishes.}
\fi

Next we discuss for the Fermat quartic and the Fermat sextic.
In fact, for all three cases, let $s$ be the parameter of $1|J^{d-1}\>\in \cH_{W_d, \<J\>}$.
Then according to \cite[Proposition 1]{LSZ}, all the ancestor FJRW functions satisfy
$$\LL\prod_{i=1}^{k}\tau_{\ell_i}(\phi_{i})\RR_{g,k}^{W_d, \<J\>}(s)\in \C[f,f',f''],$$
where $f'(s)={d\over ds}f(s)$, and $f(s)$ is the genus one FJRW function
$$f(s):=\LL 1|J^{d-1}\>\RR^{W_d, \<J\>}_{1,1}(s).$$
%and genus-zero three-point functions, which are just constants.
It satisfies a Chazy equation
$$2f'''-2f\cdot f''+3(f')^2=0$$
and therefor is completely determined by the first three coefficients, two of which vanish by the construction of the moduli space.

The remaining coefficient of the cubic case $(W_3, \<J\>)$ is calculated in \cite{LLSZ}, which completely determines the holomorphic Cayley transformation $\mathscr{C}^{\rm hol}_{\tau*}$.
Now for the Fermat quartic $W_4$ and the Fermat sextic $W_6$, the value of the remaining coefficient is not known due to the complexity of the computation.
There are two situations:
\begin{enumerate}
\item
if the coefficient is nonzero, then it determines a holomorphic Cayley transformation as in Theorem \ref{thm-lg-cy}. % and same argument leads to a proof.
\item
if the coefficient vanishes, then we have a trivial holomorphic Cayley transformation.
\end{enumerate}
In either case, Virasoro constraints hold.
\qed

\subsubsection{Additional Virasoro constraints}
Following the work of \cite{OP}, we introduce additional differential operators  for $k\geq -1$: %introduce another constraints for elliptic curve with differential operators
\begin{align*}
D_k=&-(k+1)!{\partial\over\partial s_{k+1}^0}
+\sum_{\ell\geq0}\left((\ell)_{k+1}t_{\ell}^0 {\partial\over \partial s_{k+\ell}^0}+(\ell+1)_{k+1}s_\ell^1{\partial\over \partial t_{k+\ell}^{1}}\right),\\
\bar{D}_k=&-(k+1)!{\partial\over\partial s_{k+1}^1}
+\sum_{\ell\geq0}\left((\ell)_{k+1}t_{\ell}^0 {\partial\over \partial s_{k+\ell}^1}-(\ell+1)_{k+1}s_\ell^0{\partial\over \partial t_{k+\ell}^{1}}\right).
\end{align*}
It is proved that the descendent GW potential of the elliptic curves are annihilated by these differential operators \cite[Theorem 4]{OP}.
As an application, we show that the ancestor potential $\mathcal{A}^{\rm FJRW}_{W_d,\<J\>}$ are also annihilated by these operators.
\begin{theorem}
For LG pairs $(W_d,\<J\>)$, we have additional constriants
	$$D_k \mathcal{A}^{\rm FJRW}_{W_d,\<J\>}=\bar{D}_k \mathcal{A}^{\rm FJRW}_{W_d,\<J\>}=0, \quad k\geq -1.$$
\end{theorem}
\begin{proof}
By \eqref{quadratic-calibration}, we have
\begin{equation}\label{quantization-calibration}
	(\log S(q)))^{\wedge}=-q{(\tilde{q}_0^0)^2\over 2}-q\cdot\sum_{k\geq 0} \tilde{q}^0_{k+1} {\partial \over \partial \tilde{q}^1_{k}}\,.
\end{equation}
Via dilaton shift $q^i_k=t^i_k-\delta^{i, 0}\delta_{k, 1}$, it is easy to check
\begin{align*}
	[(\log S(q))^{\wedge}, D_k]=[(\log S(q))^{\wedge}, \bar{D}_k]=0.
\end{align*}
Then
\begin{align*}
	\widehat{S}(q) e^{-\mathcal{F}^{\sE}_{1}(q)} D_k \, e^{\mathcal{F}^{\sE}_{1}(q)}\widehat{S}(q)^{-1}=\widehat{S}(q)D_k \,\widehat{S}(q)^{-1}=\exp({\rm Ad}_{(\log S(q))^{\wedge}})(D_k)=D_k.
\end{align*}
	The argument for $\bar{D}_k$ is similar. The remaining proof is the same as that of Theorem \ref{virasoro-cubic}.
\end{proof}
\iffalse
Normalized $\alpha, \beta\in H^*(E)$ such that $\alpha \wedge \beta= \omega$. Define two operators $T$ and $\bar{T}$ on the loop space $\mathbb{H}$ as,
\begin{align*}
	T:=\alpha\wedge, \quad \bar{T}:=\beta\wedge.
\end{align*}
Consider the operator $\bar{\theta}: \phi_az^k\mapsto\mu^-_a\cdot \phi_az^k$, and $\bar{\mathcal{L}}_k$ is defined the same way as $\mathcal{L}_k$ with $\bar{\theta}$, it is easy to see that
\begin{equation}
	[\mathcal{L}_k, T]=[\bar{\mathcal{L}}_k, \bar{T}]=0
\end{equation}
	So $T\circ \mathcal{L}_k$ and $\bar{\mathcal{L}}_k,\circ\bar{T}$ is infinitesimal symplectic transformation on $\mathbb{H}$. Futhermore, it can be check that
	$\bar{D}_k=\widehat{T\circ \mathcal{L}_k}$, $D_k=\widehat{\bar{\mathcal{L}}_k,\circ\bar{T}}$.
	The remaining is similar as the proof of Theorem \ref{virasoro-cubic}.
	\fi

%{\color{red}Notice that $D_k$ and$\bar{D}_k$ is slightly different with the original one in \cite[Theorem 4]{OP-virasoro}, I do not know what happen. Maybe some mistake in my calculation or some typo in \cite[Theorem 4]{OP-virasoro}. You can double check it, even if $D_k$ and$\bar{D}_k$ can not be write into quantizaiotn form, the argument in Theorem \ref{virasoro-cubic} sitll works. }

%\subsubsection{Virasoro constraints and holomorphic anomaly equations} % for orbifold simple elliptic singularities}
%The Virasoro constraints here allows us to relate two versions of holomorphic anomaly equations.

%\newpage
\subsection{Invertible CY polynomials of three variables}
\label{sec-simple-elliptic}
%In previous two sections, we consider the Virasoro constraints for LG pairs of CY type.
%When the Calabi-Yau polynomial is of three variables, we can give a complete picture.
In this section, we discuss invertible Calabi-Yau polynomials of three variables.
%\subsubsection{Classification of invertible Calabi-Yau polynomials of three variables}
According to \cite{MS}, up to permutations of variables, such a Calabi-Yau polynomial $W(x_1, x_2, x_3)$ must be one of the forms listed below.
%$W$ should be one of the form listed below.
\begin{table}[H]
\label{1dim-CY-polynomial}
\caption{Invertible CY polynomials of three variables}
\

\centering

\begin{tabular}{|c||l|l|l|}
  \hline
  $(q_1, q_2, q_3)$ & $({1\over 3},{1\over 3},{1\over 3})$ & $({1\over 4},{1\over 4},{1\over 2})$ &$({1\over 6},{1\over 3},{1\over 2})$\\
  \hline
  $W$ &  $E_6^{(1,1)}$ &    $E_7^{(1,1)}$  & $E_8^{(1,1)}$ \\
  \hline
  \hline
 Fermat & $x_1^3+x_2^3+x_3^3$ & $x_1^4+x_2^4+x_3^2$   & $x_1^6+x_2^3+x_3^2$ \\
    \hline
Chain & $x_1^2x_2+x_2^2x_3+x_3^3$ & $x_1^3x_2+x_2^2x_3+x_3^2$   &  \\
    \hline
Loop & $x_1^2x_2+x_2^2x_3+x_3^2x_1$ &  & \\
   \hline
Mixed & $x_1^3+x_2^2x_3+x_3^3;$ & $x_1^3x_2+x_2^4+x_3^2;$   & $x_1^4x_2+x_2^3+x_3^2;$ \\
&$x_1^2x_2+x_1x_2^2+x_3^3$& $x_1^4+x_2^2x_3+x_3^2;$   & $x_1^3+x_2^3x_3+x_3^2$ \\
 &  & $x_1^3x_2+x_2^3x_1+x_3^2$   &\\
  \hline
\end{tabular}
\end{table}

%$$\widehat{c}_W=\sum_{i=1}^{n}(1-2q_i)=n-2=1.$$
%\subsubsection{%Classification of %simple elliptic singularities  Invertible CY singularities with three variables}
%Assume $W$ has three variable (use stablelization). Explain why.
%It is not hard to list all the forms of $W$.

Here $(q_1, q_2, q_3)$ are the weights of the variables.
The notation $E_{\mu-2}^{(1,1)}$ is introduced by Saito for the {\em simple elliptic singularities} \cite{Sai-ses}, where $\mu$ is the Milnor number of the singularity.
In \cite{AGV}, they are denoted by $P_8, X_9$, and $J_{10}$.
%In deformation theory of singularities, the weight-${1\over 2}$ variables can be removed by stabilization.
%If $W$ is invertible and $\widehat{c}_W=1$, then we abuse the notation to call $W$ a {\em simple elliptic singularity}.

%In particular, the polynomial $W$ in the pillowcase LG pair in \eqref{pillow-lg} is a simple elliptic singularity of type $E_7^{(1,1)}$.

For such a polynomial $W(x_1, x_2, x_3)$, and an arbitrary group $G$ such that $\<J\>< G < G_W$, using Theorem \ref{semi-simple-Virasoro} and Theorem \ref{virasoro-cubic}, the Virasoro Conjecture \ref{VC} for the LG pairs $(W(x_1, x_2, x_3), G)$ of CY type will follow from the conjecture below.
\begin{conjecture}
\label{simple-elliptic-conjecture}
Let $W$ be an invertible Calabi-Yau polynomial of three variables.
The LG A-model theory of an admissible LG pair $(W, G)$ is either generically semisimple, or equivalent to the LG A-model theory of a Fermat Calabi-Yau pair
$$(W=x_1^{a_1}+x_2^{a_2}+x_3^{a_3}, \<J_W\>).$$
\end{conjecture}

This conjecture can be checked case-by-case in two steps.
\begin{enumerate}
\item  The bigraded superspace $\cH_{W,G}$ is isomorphic to one of the five cases:
$$
\begin{dcases}
\cH_{W, G_W}, & W=x_1^{a_1}+x_2^{a_2}+x_3^{a_3}, \quad (a_1, a_2, a_3)=(3, 3,3), (4,4,2), (6,3,2);\\
\cH_{W, \<J_W\>}, & W=x_1^{3}+x_2^{3}+x_3^{3};\\
\cH_{W, \<J_W\>}, & W=x_1^4+x_2^4.
\end{dcases}
$$

\item
In these cases,  the ranks of the state spaces are $4,6,8,9,10$ respectively.
The first three cases are special cases in Theorem \ref{thm-max}.
The fourth case is discussed in Theorem \ref{virasoro-cubic}.
The last case is a special case in Proposition \ref{general-pillow}.
%The last case is a special case of $(x^d+y^d, \<J\>)$, it will be proved in \cite{FHS}.
%in Proposition \ref{two-fermat-semisimple}.
%So Conjecture \ref{simple-elliptic-conjecture} can be checked similarly.
%The Frobenius manifolds of the last four are all generically semisimple.
%We can discuss similarly.
\end{enumerate}

\begin{example}
For the quartic polynomial $W=x_1^4+x_2^4+x_3^2$, we consider the following choices of admissible groups
$$
\begin{tikzcd}[column sep=tiny]
& \<J_{x_1^4+x_2^4}\>\times \<J_{x_3^2}\> \ar[dr, hook] &\\
\<J_W\> \arrow[ur, hook] \arrow[dr, hook] & & G_W \\
& {\rm SL}(3, \C)  \arrow[ur, hook]  &
\end{tikzcd}
$$
Here are some observations:
\begin{itemize}
\item
The state spaces $\cH_{W, G}$ are isomorphic if
$G=\<J_W\>$ or $G={\rm SL}(3, \C)$.
\item
By \cite[Theorem 4.1.8 (8)]{FJR},  the FJRW theory of the pair $$(W, G=\<J_{x_1^4+x_2^4}\>\times \<J_{x_3^2}\>)$$ is isomorphic to the FJRW theory of the pair
$$(x_1^4+x_2^4, \<J_{x_1^4+x_2^4}\>).$$ %The later has been studied in Section \ref{sec-nonmax}.
\end{itemize}
\end{example}

We call $(W, G)$ a {\em pillowcase LG pair} if the orbifold curve $(W=0)/\widetilde{G}$ on the Calabi-Yau side is the {\em pillowcase} $\mathbb{P}^1_{2,2,2,2}$ \cite{EO}.
%The name is borrowed from \cite{EO} where $\mathbb{P}^1_{2,2,2,2}$ is called the {\em pillowcase}.
For example, the pair
$$(W=x_1^4+x_2^4+x_3^2, G=\<J_{x_1^4+x_2^4}\>\times \<J_{x_3^2}\>)$$ is a {\em pillowcase LG pair}.
For pillowcase LG pairs, a stronger version of Conjecture \ref{simple-elliptic-conjecture} would be the LG A-model theories of all pillowcase LG pairs have isomorphic generically semisimple CohFTs.
\appendix

\section{A super trace formula}
\label{sec-super} % for the admissible LG pairs}
In this section, we prove Proposition \ref{conj-supertrace}.
Our proof is a minor modification of the elegant argument in \cite{ET}, where the cases of $G\subseteq {\rm SL}(n, \C)$ is proved.
%However, if one carefully trace their argument, one will find that
The condition $G\subseteq {\rm SL}(n, \C)$ is only used to define the ``variance" ${\rm Var}_{(f, G)}$. The variance is equivalent to the super trace ${\rm Str}(\theta^2)$ if $G\subseteq {\rm SL}(n, \C)$. If we use ${\rm Str}(\theta^2)$, the condition $G\subseteq {\rm SL}(n, \C)$ can be removed. We will explore their idea here.

Firstly, fix $W$, we define the Poincar\'e series of $\Jac(W)\cdot d\bx$ as following:
\begin{equation}\label{poincare}
	P_{W}(y):=\sum_{\alpha\in \Jac(W)\cdot \Omega}(-1)^n \cdot y^{{\rm wt}(\alpha)}
\end{equation}
Recall $q_i$ is the weight of $x_i$ in $W$. The following is a standard formula due to A. G. Kouchnirenko \cite{Ko} and  J. Steenbrink \cite{St},
\begin{lemma}
The Poincar\'e series is
\begin{equation}\label{Poinare for}
	P_{W}(y)=\prod_{i=1}^n\frac{y-y^{q_i}}{1-y^{q_i}}.
\end{equation}
\end{lemma}
\begin{proof}
%\emph{Proof of \eqref{Poinare for}.}
Let $A:=\C[x_1,\ldots, x_n],$ and
$$K_j:= \bigoplus_{|I|=j} A\cdot dx_I$$
be the $A$-module consist of $j$-differential form in $\C^n$.
We consider the following Koszul complex:
\begin{equation}\label{Koszul}
	0\rightarrow K_0\xrightarrow{d_1} K_1\xrightarrow{d_2} \cdots \xrightarrow{d_n} K_n\xrightarrow{\pi} \Jac(W)\cdot d\bx \rightarrow 0
\end{equation}
where $\pi$ is the projection, $d_i$ is the differential defined by
$$d_i:=\sum_{i=1}^n \frac{\partial W}{\partial x_i}\cdot dx_i\wedge.$$
Recall that $W$ is a non-degenerate polynomial, with isolated critical points only at the origin, so the complex \eqref{Koszul} is exact. Furthermore, assign $q_i$ as the degree of $x_i$ and $dx_i$, we know $d_j$ is an operator of degree 1 for all $j$.
Similar as \eqref{poincare}, we define the Poincar\'e series of $A$ and $K_j$:
\begin{align*}
P_A(y)&=\prod_{i=1}^n\frac{1}{y^{q_i}-1},
\quad \\
P_{K_j}(y)&=P_A(y)\cdot \sum_{|I|=j}\prod_{i\in I}y^{q_i}.
\end{align*}
Via the exactness of \eqref{Koszul} and the fact that $d_i$ is degree 1, we have
\begin{align*}
	P_{W}(y)&=P_{K_n}(y)-y\cdot P_{K_{n-1}}(y)+\cdots (-1)^ny^n P_{K_0}(y)\\
	&=\prod_{i=1}^n\frac{y-y^{q_i}}{1-y^{q_i}}.
\end{align*}
\end{proof}

Now let $(W, G)$ be an admissible LG pair.
For $g\in G$, we write its diagonal action on $\C^n$ by
$$g: (x_1, \ldots, x_n )\mapsto (\lambda_1(g) x_1, \ldots, \lambda_n(g)x_n ).$$
It induces a $g$-action on each $\alpha\in \Jac(W)\cdot d\bx$, denoted by
$$g\cdot \alpha =\rho_g (\alpha) \cdot \alpha.$$
We define the Poincar\'e series of $\Jac(W)\cdot d\bx$ coupled with $g$ as
\begin{equation}
	P_{W, g}(y):=\sum_{\alpha\in \Jac(W)\cdot d\bx}(-1)^n \rho_g (\alpha)\cdot y^{{\rm wt}(\alpha)}.
\end{equation}
Notice that $g$ also acts on $A$ and $K_j$ naturally, and $d_j$ is $g$-invariants as $W$ is $g$-invariant. Via a similar argument as above, we have
\begin{align*}
	P_{A, g}(y)&=\prod_{i=1}^n\frac{1}{\lambda_i(g)y^{q_i}-1}, \quad\\
	P_{W, g}(y)&=\prod_{i=1}^n\frac{y-\lambda_i(g)y^{q_i}}{1-\lambda_i(g)y^{q_i}}.
\end{align*}
Then summing over all $g\in G$, the following formula for the Poincar\'e series of the $G$-invariant part of $\Jac(W)\cdot d\bx$ holds \cite[Theorem 6]{ET}:
\begin{align}
P_{W, G}(y):=\sum_{\alpha\in (\Jac(W)\cdot d\bx)^G}(-1)^n \cdot y^{{\rm wt}(\alpha)}
={1\over |G|}\sum_{g\in G}\prod_{i=1}^n\frac{y-\lambda_i(g)y^{q_i}}{1-\lambda_i(g)y^{q_i}}.\label{G-Poinare}
\end{align}

%Now for an admissible LG pair $(W, G)$, consider the state space
%$$H_{W, G}=\bigoplus_{\gamma\in G}\left({\rm Jac}(W_\gamma)\cdot \Omega_\gamma\right)^G$$ and the bigrading $(\mu^+, \mu^-)$,

We define its Poincar\'e series by
\begin{equation*}
	P_{(W, G)}(y):=\sum_{\phi_a\in \cH_{W, G}}(-1)^{|\phi_a|} \cdot y^{\mu^+_a}.
\end{equation*}
By definition of $\cH_{W, G}$, the bigrading and parity, we have
\begin{align}\label{G-P}
		P_{(W, G)}(y)&=\sum_{\gamma\in G}y^{\iota_\gamma-{\widehat{c}_W\over 2}}\cdot P_{W_\gamma, G}(y)\\
		&=\sum_{\gamma\in G}y^{{\rm age}(\gamma)-{n\over 2}}\cdot {1\over |G|}\sum_{g\in G}\prod_{\substack{1\leq  i\leq n,\\ \lambda_i(\gamma)=1}}\frac{y-\lambda_i(g)y^{q_i}}{1-\lambda_i(g)y^{q_i}}\\
		&=\sum_{\gamma\in G}y^{{\rm age}(\gamma)-{n-N_\gamma\over 2}}\cdot {1\over |G|}\sum_{g\in G}\prod_{\substack{1\leq  i\leq n, \\ \lambda_i(\gamma)=1}}\frac{y^{1\over 2}-\lambda_i(g)y^{q_i-{1\over 2}}}{1-\lambda_i(g)y^{q_i}}\nonumber
\end{align}
%Here ${\rm age}(\gamma)=\iota_\gamma+\sum_{i=1}^nq_i$.
Notice that in the above equation, we use \eqref{G-Poinare} and the fact that for all $\gamma\in G$, $W_\gamma$ is still a a non-degenerate quasi-homogeneous polynomial %with isolated critical points only at the origin
\cite[Lemma 2.1.10]{FJR} \cite[Proposition 5]{ET}.
It is easily to see
\begin{lemma}
Let $(W, G)$ be an admissible LG pair, we have
\begin{align*}
\chi_{W,G}&=\lim_{y\rightarrow 1}P_{(W, G)}(y), \quad \\
{\rm Str}(\theta^2)&=\lim_{y\rightarrow 1}{d\over dy}\left(y{d\over dy}P_{(W, G)}(y)\right).
\end{align*}
\end{lemma}
Now it suffices to prove:
\begin{proposition}
Let $(W, G)$ be an admissible LG pair, we have
	$$\lim_{y\rightarrow 1}{d\over dy}\left(y{d\over dy}P_{(W, G)}(y)\right)={\widehat{c}_W\over 12}\cdot \lim_{y\rightarrow 1}P_{(W, G)}(y).$$
\end{proposition}
%{\red Say something about the proof. Direct calculation? which trick is essential? Geometric understanding?}
In fact, from \eqref{G-P}, we know $P_{(W,G)}(y)$ here coincides with $(-1)^n\chi(W, G)(y)$ defined in \cite{ET}, then the above Proposition is exactly \cite[Theorem 19]{ET}.
We omit the detail here, only mention that the proof there works for all $G\subseteq G_W$.

\section{Genus 0 modified Virasoro constraints}
As we see in Section \ref{A grading equation and the operator $L_0$}, an obstruction in the proof of $L_0$-constraints is the grading assumption \ref{alge cycle}, of which the proof is lacked. However, in FJRW theory or KL theory, the following degree constrain holds :
\begin{lemma}\label{selection rule}
For $\clubsuit=$FJRW or KL, the quantum invariant
   $$\LD\prod_{i=1}^{k}\tau_{\ell_i}(\phi_{i})\RD^{\clubsuit, W, G}_{g, k}\neq 0$$
   only if
   $$\sum_{i=1}^k\left({\mu_{i}^{+}+\mu_{i}^{-}\over 2}+{\widehat{c}_W\over 2}+\ell_i\right)=(3-\widehat{c}_W)(g-1)+k.
   $$
\end{lemma}
\begin{proof}
$2\deg_{\C}\phi_i$ is the cohomological degree(after degree shift) of  Lefschetz thimbles in FJRW theory \cite{FJR}, and that of intersection homology in KL theory \cite{KL}. This lemma is a corollary of \eqref{complex-degree} and  \eqref{deg-fjrw}.
\end{proof}
Now we modify the genus 0 constraints as following(just replace $\mu_{i}^+$ with ${\mu_{i}^{+}+\mu_{i}^{-}\over 2}$ in Definition \ref{def-virasoro}, and restrict it to genus 0 case)
\begin{definition}
\label{def-virasoro}
For each integer $k\in\mathbb{Z}_{\geq -1}$, we introduce a differential operator
\begin{align}
L^M_k:=&-\left(\frac{3-\widehat{c}_W}{2}\right)_{k+1} {\pop t_{k+1}^{0}}\nonumber\\
%-\frac{\Gamma(\frac{5-c_W}{2}+k)}{\Gamma(\frac{3-c_W}{2})}\partial_{k+1, 0}\nonumber\\
&+\sum_{m=0}^\infty\left({\mu_{a}^{+}+\mu_{a}^{-}\over 2}+m+\frac{1}{2}\right)_{k+1} t^a_m{\pop t_{m+k}^{a}}\nonumber\\
%&+\sum_{m=0}^\infty\frac{\Gamma(\mu^+_a+m+k+\frac{3}{2})}{\Gamma(\mu^+_a+m+\frac{1}{2})}t^a_m\partial_{m+k, a}\nonumber\\
&+\frac{\hbar^2}{2}\sum_{m=-k}^{-1}(-1)^m\left({\mu_{a}^{+}+\mu_{a}^{-}\over 2}+m+\frac{1}{2}\right)_{k+1}  \eta^{ab}{\pop t_{-m-1}^{a}}{\pop t_{m+k}^{b}}\nonumber\\
%\label{virasoro-operator}\\
&+\frac{1}{2\hbar^2}\delta_{-1,k}\eta_{ab}t_0^at_0^b\nonumber\\
&-{\delta_{0, k}\over 4}\sum_{a}(-1)^{|\phi_a|}({\mu_{a}^{+}+\mu_{a}^{-}\over 2}-{1\over 2})({\mu_{a}^{+}+\mu_{a}^{-}\over 2}+{1\over 2}).\nonumber
%+\delta_{0, k}\cdot {\chi\over 24}\cdot \left({3-\widehat{c}_W\over 2}\right).\nonumber
\end{align}
\end{definition}
Define the genus 0 potential as
\begin{equation}
\mathcal{F}_0^{\clubsuit}(\mathbf{t})
:=\sum_{k}{1\over k!}\LD\prod_{i=1}^{k}\bt(\psi_i)
%\bt(\psi_1), \cdots, \bt(\psi_k)
\RD_{0,k}^{\clubsuit, (W, G)}.
\end{equation}
One can show that $L^M_0\exp\left({\hbar^{-2}\mathcal{F}_0^{\rm FJRW, KL}}\right)\in \C[[\hbar]]\cdot\exp\left({\hbar^{-2}\mathcal{F}_0^{\rm FJRW, KL}}\right)$ via the proof in Section \ref{A grading equation and the operator $L_0$}. Furthermore, via the same argument in \cite{LT} and \cite{Ge} , we have
\begin{theorem}
\begin{equation}
L^M_k\exp\left({\hbar^{-2}\mathcal{F}_0^{\rm FJRW, KL}}\right)\in \C[[\hbar]]\cdot\exp\left({\hbar^{-2}\mathcal{F}_0^{\rm FJRW, KL}}\right), \quad k\geq-1.
\end{equation}
\end{theorem}

  Since the coefficients of $\hbar^{-2}$ in $L^M_k\exp\left({\hbar^{-2}\mathcal{F}_0^{\rm FJRW, KL}}\right)/\exp\left({\hbar^{-2}\mathcal{F}_0^{\rm FJRW, KL}}\right)$ vanishes, one can obtain a series of recursive relations for genus 0 invariants.

\begin{remark}
In general, one can check $L^M_k$ can not be a constraints of $\mathcal{A}^{\clubsuit}_{W,G}$ for higher genus. In fact, there does not exist a constant $c$ such that $L^M_0+c$ behave well, which means that, $(L^M_0+c)\mathcal{A}^{\clubsuit}_{W,G}=0$ and Virasoro relation \eqref{Virasoro-relation} holds simultaneously. This is because of the absence of supertrace formula \eqref{hodge-rr} in this case.  See \cite[Section 2.10]{Ge} for this issue for Gromov-Witten theory.
\end{remark}

\vskip .2in
\noindent

%\newpage

\vskip .2in

\noindent{\small Department of Mathematics, Sun Yat-sen University, Guangzhou, Guangdong 510275,
China}

\noindent{\small E-mail: hewq@mail2.sysu.edu.cn}

\vskip .1in

\noindent{\small Department of Mathematics, University of Oregon, Eugene, OR 97403,
USA}

\noindent{\small E-mail: yfshen@uoregon.edu}

%\vskip .1in

%\noindent{\small Mathematical Reviews, American Mathematical Society,   Ann Arbor, Michigan 48103, USA}

%\noindent{\small E-mail: aef@ams.org}

\end{document}